\documentclass[11pt, reqno]{amsart}
\usepackage[colorlinks,citecolor=red,pagebackref,hypertexnames=false]{hyperref}

\usepackage[T1]{fontenc}

\usepackage{textcomp}

\large

\makeatletter
\def\@tocline#1#2#3#4#5#6#7{\relax
	\ifnum #1>\c@tocdepth 
	\else
	\par \addpenalty\@secpenalty\addvspace{#2}%
	\begingroup \hyphenpenalty\@M
	\@ifempty{#4}{%
		\@tempdima\csname r@tocindent\number#1\endcsname\relax
	}{%
		\@tempdima#4\relax
	}%
	\parindent\z@ \leftskip#3\relax \advance\leftskip\@tempdima\relax
	\rightskip\@pnumwidth plus4em \parfillskip-\@pnumwidth
	#5\leavevmode\hskip-\@tempdima
	\ifcase #1
	\or\or \hskip 1em \or \hskip 2em \else \hskip 3em \fi%
	#6\nobreak\relax
	\dotfill\hbox to\@pnumwidth{\@tocpagenum{#7}}\par
	\nobreak
	\endgroup
	\fi}
\makeatother

\usepackage{tikz,amsthm,amsmath,amstext,amssymb,amscd,epsfig,euscript, mathrsfs, dsfont,pspicture,multicol, mathtools,graphpap,graphics,graphicx,times,enumerate,subfig,sidecap,wrapfig, color,array}

\usepackage{ hyperref}
\usepackage{enumerate}
\usepackage{esint}
\usepackage{verbatim} 
\usepackage[many]{tcolorbox}
\usepackage{dsfont}

\numberwithin{equation}{section}

\def\R{{\mathbb{R}}}

\newcommand{\dA}{\mathcal{A}}
\newcommand{\dG}{\mathcal{G}}
\newcommand{\dM}{\mathcal{M}}
\newcommand{\dD}{\mathcal{D}}
\newcommand{\dH}{\mathcal{H}}
\newcommand{\dW}{\mathcal{W}}

\newcommand{\dT}{\mathcal{T}}
\newcommand{\dN}{\mathcal{N}}
\newcommand{\dS}{\mathcal{S}}
\newcommand{\dV}{\mathcal{V}}

\newcommand{\dB}{\mathcal{B}}

\newcommand{\Nw}{N_{\text{\tiny{W}}}}
\newcommand{\bV}{\mathbb{V}}
\newcommand{\bY}{\mathbb{Y}_{{\text{\tiny{W}}}}}
\newcommand{\Gaw}{\overline{\mathbb{Y}}_{{\text{\tiny{W}}}}}

\newcommand{\Grad}{{\rm Grad}}

\newcommand{\wt}[1]{\widetilde{#1}}
\def\ve{\varepsilon}
\def\vp{\varphi}

\renewcommand{\d}{{\partial}}
\def\lec{\lesssim}

\newcommand{\vvv}{{\vspace{2mm}}}

\newcommand{\pom}{{\partial \Omega}}

\newcommand{\om}{{\Omega}}

\newcommand{\WW}{{\mathcal W}}
\newcommand{\DD}{{\mathcal D}}
\newcommand{\s}{{\sigma}}

\newcommand{\NN}{{\mathcal N}}

\newcommand{\bwgl}{\mathsf{BWGL}}
\newcommand{\cw}{\mathsf{U}}

\DeclareMathOperator{\diam}{diam}
\def\BMO{\mathop\mathrm{BMO}} 					
\def\Lip{\mathop\mathrm{Lip}} 						
\def\dist{\mathop\mathrm{dist}} 						
\def\supp{\mathop\mathrm{supp}}					
\DeclareMathOperator*{\esssup}{ess\,sup}			

\newcommand{\ps}[1]{\left( #1 \right)}

\def\M{\mathcal{M}}

\def\Xint#1{\mathchoice
	{\XXint\displaystyle\textstyle{#1}}%
	{\XXint\textstyle\scriptstyle{#1}}%
	{\XXint\scriptstyle\scriptscriptstyle{#1}}%
	{\XXint\scriptscriptstyle\scriptscriptstyle{#1}}%
	\!\int}
\def\XXint#1#2#3{{\setbox0=\hbox{$#1{#2#3}{\int}$ }
		\vcenter{\hbox{$#2#3$ }}\kern-.58\wd0}}

\def\avint{\Xint-}
\def\grad{\nabla}

\theoremstyle{maintheorem}
\newtheorem{maintheorem}{Theorem}

\theoremstyle{plain}
\newtheorem{theorem}{Theorem}

\newtheorem{corollary}[theorem]{Corollary}

\newtheorem{lemma}[theorem]{Lemma}

\newtheorem{proposition}[theorem]{Proposition}

\newtheorem{claim}[theorem]{Claim}

\theoremstyle{definition}

\newtheorem{definition}[theorem]{Definition}
\newtheorem{remark}[theorem]{Remark}

\numberwithin{equation}{section}
\numberwithin{theorem}{section}

\DeclareFontFamily{U}{mathb}{\hyphenchar\font45} 
\DeclareFontShape{U}{mathb}{m}{n}{
	<5> <6> <7> <8> <9> <10> gen * mathb
	<10.95> mathb10 <12> <14.4> <17.28> <20.74> <24.88> mathb12
}{}
\DeclareSymbolFont{mathb}{U}{mathb}{m}{n}

\DeclareMathSymbol{\toitself}{3}{mathb}{"FD}  

\def\salpha{\alpha_{\sigma}^{d}}

\makeatletter
\def\@tocline#1#2#3#4#5#6#7{\relax
	\ifnum #1>\c@tocdepth 
	\else
	\par \addpenalty\@secpenalty\addvspace{#2}%
	\begingroup \hyphenpenalty\@M
	\@ifempty{#4}{%
		\@tempdima\csname r@tocindent\number#1\endcsname\relax
	}{%
		\@tempdima#4\relax
	}%
	\parindent\z@ \leftskip#3\relax \advance\leftskip\@tempdima\relax
	\rightskip\@pnumwidth plus4em \parfillskip-\@pnumwidth
	#5\leavevmode\hskip-\@tempdima
	\ifcase #1
	\or\or \hskip 1em \or \hskip 2em \else \hskip 3em \fi%
	#6\nobreak\relax
	\dotfill\hbox to\@pnumwidth{\@tocpagenum{#7}}\par
	\nobreak
	\endgroup
	\fi}
\makeatother

\renewcommand{\Sigma}{E}

\begin{document}
	\setcounter{tocdepth}{1}
	\setcounter{secnumdepth}{4}

	\title[Smooth extensions]{Smooth extensions of Sobolev boundary data in corkscrew domains with uniformly rectifiable boundaries}
	
	\author[Azzam, Mourgoglou and Villa]{Jonas Azzam, Mihalis Mourgoglou and Michele Villa}
	\address{Jonas Azzam\\
		School of Mathematics \\ University of Edinburgh \\ JCMB, Kings Buildings \\
		Mayfield Road, Edinburgh,
		EH9 3JZ, Scotland.}
	\email{j.azzam "at" ed.ac.uk}
	
	\address{Mihalis Mourgoglou\\
		Departamento de Matem\'aticas, Universidad del Pa\' is Vasco, Barrio Sarriena s/n 48940 Leioa, Spain and\\ Ikerbasque, Basque Foundation for Science, Bilbao, Spain.} \email{michail.mourgoglou "at" ehu.eus}
	
	\address{Michele Villa\\
		Departamento de Matem\'aticas, Universidad del Pa\' is Vasco, Barrio Sarriena s/n 48940 Leioa, Spain and\\ Ikerbasque, Basque Foundation for Science, Bilbao, Spain.}
		\email{michele.villa "at" ehu.eus}
	

	\thanks{J.A. was partially supported by Basque Center for Applied Mathematics (BCAM) while on a research visit to M.M. 
		M.M. was supported  by IKERBASQUE and partially supported by the grant PID2020-118986GB-I00 of the Ministerio de Econom\'ia y Competitividad (Spain) and by the grant IT-1615-22 (Basque Government). M.V. was supported by a starting grant of the University of Oulu and by the Academy of Finland via the project
		“Higher dimensional Analyst’s Traveling Salesman theorems and Dorronsoro estimates on non-smooth sets”, grant
		No. 347828/24304228}
	


	\begin{abstract}
		Given a corkscrew domain with uniformly rectifiable boundary, we construct a surjective trace map onto the $L^p$ Haj\l{}asz-Sobolev space on the boundary from the space of functions on the domain with $L^p$ norm involving the non-tangential maximal function of the gradient and the conical square function of the Hessian. This fundametally uses the Dorronsoro theorem for UR sets proven in the companion paper \cite{AMV1}.
		
	\end{abstract}
	\maketitle
	\vspace{-1cm}
	\tableofcontents

\section{Introduction}\label{s:corollary}
Given an open set $\Omega \subset \R^{d+1}$, a `domain' function space $X(\om)$ and a `boundary' function space $Y(\pom)$, it is often important to understand the trace map
\[
\mathcal{T} : X(\om) \to Y(\pom).
\]
We face two different issues. First, we would like to check that the trace map in fact exists from $X(\om)$ to $Y(\pom)$, and that it has good norm bounds
\[
\|\dT u\|_{Y(\pom)} \lesssim \|u\|_{X(\om)}.
\]
Second, we want to know whether the trace map is surjective--that is, for any given $f \in Y(\pom)$, can we find $u \in X(\om)$ such that $\mathcal{T}u = f$? This reduces to showing an estimate of the type
\[
\|u\|_{X(\om)} \lesssim \|f\|_{Y(\pom)},
\]
where $u$ is a function \textit{constructed} from $f$, or in other words, an extension. Our Theorem \ref{t:extension-in-text} below provides a bounded and surjective trace map $\mathcal{T}$ onto $M^{1,p}(\pom)$ from the space of functions $u$ on $\om$ whose non-tangential maximal function of the gradient lies in $L^p$. To be more precise, we introduce some notation. Set $d_\Omega(x) = \dist(x,\pom)$. For $\alpha > 0$, let
\begin{equation*}
	\Gamma(x)= \left\{  y \in \Omega: \tfrac{1}{2}\cdot |x-y| < d_\Omega(y)  \right\}.
\end{equation*}
\noindent
For a vector field $F: \om \to \R^m$, $1 \leq m<+\infty$, the non-tangential maximal function of $F$ is given by 
\begin{equation*}
	\dN (F)(x)= \sup_{y \in \Gamma(x) } |F(y)|.
\end{equation*}
Then set 
\begin{equation*} 
	\dV_{p}(\om) := \left\{u \in C^1(\om) \, |\, \dN(\grad u) \in L^p(\pom)\right\}.
\end{equation*}
For later use, we also introduce the conical square function:
\begin{equation*}
	\dS (F)(x)= \left( \int_{\Gamma(x)} |F(y)|^2\,\frac{dy}{\delta_\om(y)^{d+1}} \right)^{1/2},
\end{equation*}
and the corresponding space
\begin{equation*}
	\dV_{p}^S (\om) := \left\{ u \in C^2(\om) \, |\, \dN(\grad u) \in L^p(\pom)\, \mbox{ and }\, \dS(d_\Omega(\cdot)\nabla^2 u) \in L^p(\pom)\right\}.
\end{equation*}

\noindent
Many of the conclusions in our theorem below will hold for corkscrew domains with uniformly $d$-rectifiable boundary.
\begin{definition}
	Following \cite{jerison1982boundary}, we say that an open subset $\om\subset \R^{d+1}$ satisfies the corkscrew condition, or that it is a corkscrew open set (or domain), if there exists a constant $c>0$ such that for all $x \in \pom$ and all $r \in (0, \diam(\om))$, there exists a ball $B \subset B(x,r) \cap \om$ so that $r(B) \geq c r$.
\end{definition}

However, for the trace map to satisfy good bounds we need some quantitative connectedness of the domain. We choose to assume that our domain satisfies the local John condition. This condition is rather weak, but we currently don't know if it is sharp.

\begin{definition}
	Fix $0<\theta \leq 1$. Let $x,y \in \overline{\om}$. We say that a rectifiable curve $\gamma:[0,1] \to \Omega$ is a $\theta$-\textit{carrot path connecting $x$ to $y$} if
	\begin{enumerate}
		\item $\gamma(0)=x$ and $\gamma(1)=y$. 
		\item $\mathcal{H}^1(\gamma([0,1]) \leq \theta^{-1}|x-y|$. 
		\item $\theta \cdot |\gamma(t)-x| < \delta_\om(\gamma(t))$ for all $t \in (0, 1]$.
	\end{enumerate}
\end{definition}

\begin{definition}\label{d:ljc}
	We say that a bounded open subset $\Omega \subset \R^{d+1}$ satisfies the \textit{local John condition} (or that $\Omega$ is a local John domain, LJD) if there exists a $\theta>0$ such that the following holds true. For all $\xi \in \pom$ and $0<r<\diam(\Omega)$, we can find a point $x_\xi \in B(x, r)\cap \om$ such that 
	\begin{enumerate}
		\item $B(x_{\xi}, \theta r) \subset \om$;
		\item for each $\xi' \in \pom \cap B(\xi, r)$ there is a $\theta$-carrot path connecting $\xi'$ to $x_{\xi}$. That is, we can find a rectifiable path $\gamma_{\xi'}=\gamma_{\xi', x_\xi}:[0,1] \to \overline{\Omega}$ of length $\leq \theta^{-1} r$ and such that
		\begin{gather*}
			\gamma_{\xi'}(0) = \xi', \quad \gamma_{\xi'}(1) = x_{\xi},\\
			\mbox{and } \, \dist(\gamma_{\xi'}(t), \pom)> \theta |\gamma_{\xi'}(t)-\xi'| \,\mbox{ for all }\, t>0.
		\end{gather*}
	\end{enumerate}
\end{definition}
Our result reads as follows.

\begin{maintheorem}\label{t:extension-in-text}
	Let $d \geq 1$, let $1<p<\infty$ and let $\Omega \subset \R^{d+1}$ be a corkscrew domain with uniformly $d$-rectifiable boundary $\pom$. Then the following holds.
	\begin{enumerate}
		\item For any $f \in M^{1,p}(\pom)$ we can construct an extension $u$ of $f$ so that $u \in \dV_{p}^S$ with the bound
		\begin{equation}\label{e:extension-Lp-est}
			\|\NN(\nabla u)\|_{L^p(\pom)}+  \| \dS(d_\om(\cdot)\, \nabla^2 u) \|_{L^p(\pom)} \lesssim \| g \|_{L^{p}(\pom)},
		\end{equation}
		for any $g$ which is a Haj\l{}asz $p$-upper gradient of $f$.
		
		\item The trace $\dT u$ of any $u \in \dV_{p}$ is well defined $\sigma$-almost everywhere in the sense of Whitney averages.
		\item 
		If, on top of the current hypotheses, we assume that $f \in \Lip(\pom)$, then
		\begin{equation}\label{e:lip-ext}
			u \mbox{ is Lipschitz on } \overline{\om},
		\end{equation}
		and
		\begin{equation}\label{ext-limit}
			\nabla u \to \nabla_t f(x)\,\,\textup{non-tangentially}\,\, \s\textup{-a.e}. \,\,x \in \pom,
		\end{equation}
		where $\nabla_t f$ stands for the tangential gradient of $f$.
	\end{enumerate}
	If $\Omega$ is a local John domain with constant $\theta>0$ and uniformly $d$-rectifiable boundary then the following also holds.
	\begin{enumerate}
		\item[(4)] The trace map $\dT: \dV_{p}(\om) \to M^{1,p}(\pom)$ is well defined with norm bound
		\begin{equation}\label{e:trace-bound}
			\|\dT u\|_{\dot M^{1,p}(\pom)} \lesssim_{p, \theta} \|\dN(\grad u)\|_{L^{p}(\pom)}.
		\end{equation}
		\item[(5)] The trace map $\dT$ is surjective, and given any $f \in M^{1,p}(\pom)$ there exists a function $u\in \dV_p$ so that $\dT u = f$, so, in particular, $u$ coincides with $f$ $\sigma$-a.e. on $\partial \Omega$ in the sense of non-tangential convergence of Whitney averages.
		\noindent

	\end{enumerate}
\end{maintheorem}

\noindent
Let us survey some recent literature - but mind that we will just skim the surface of a very broad and well studied area) In fact, we will mostly focus on the literature from the `UR world'.\\

\noindent

Motivated by the corona problem in higher dimensions, Varopoulos \cite{varopoulos1977bmo, varopoulos1978remark} proved that ${\rm BMO}(\mathbb{R}^d)$ can be characterized by the fact that each $f\in{\rm BMO}$ in this space can be extended to a function $F$ on $\mathbb{R}^{d+1}_+=\mathbb{R}^d\times\mathbb{R}_+$ so that $|\nabla F|\, dt\, dx$ is a Carleson measure. A main tool in Varopoulos’ argument was an \textit{$\ve$-approximability} result, stating that a bounded analytic function in the upper half-plane can be $\ve$-approximated by a $C^\infty$ function whose gradient defines a Carleson measure (see also Theorem 6.1, Chapter VIII in \cite{garnett2007bounded}). 

Fast-forward forty years, and we find that the $\ve$-approximability of bounded harmonic functions in fact characterizes corkscrew domains with UR boundary \cite{hofmann2016uniform, garnett2018uniform}. In 2018, Hyt\"onen and Ros\`en introduced an $L^p$ version of Varopoulos’ $\ve$-approximability: they showed that \textit{any} weak solution to certain elliptic partial differential equations on $\mathbb{R}^{d+1}_+$ is $\ve$-approximable in their $L^p$ sense (\cite[Theorem 1.3]{hytonen2018bounded})--Varopoulos’ original notion concerned only \textit{harmonic} functions. They established the same result for dyadic martingales (see \cite[Theorem 1.2]{hytonen2018bounded}), and used this to construct a bounded and surjective trace map onto $L^p(\mathbb{R}^d)$ from a space of functions $u$ of locally bounded variation on the half-space $\mathbb{R}^{d+1}_+$, satisfying $\|C(\nabla u)\|_p < \infty$ and $\|\mathcal{N}u\|_p < \infty$. Here $C$ is the Carleson functional
\[
C\mu(x) := \sup_{Q\ni x} \frac{1}{|Q|} \int_{\widehat{Q}} d|\mu|(x,t),
\]
where the supremum is taken over dyadic cubes in $\mathbb{R}^d$ and $|\mu|$ is a locally finite measure on $\mathbb{R}^{d+1}_+$; $\mathcal{N}$ denotes the non-tangential maximal function. 

Shortly thereafter, it was shown that the $L^p$ notion of $\ve$-approximability (for harmonic functions) is equivalent to uniform rectifiability of the boundary of a corkscrew domain with Ahlfors regular boundary \cite{hofmann2020uniform, bortz2019approx}. Hofmann and Tapiola \cite{hofmann2021uniform} showed that if $\Omega$ is a corkscrew domain with UR boundary, then one can construct Varopoulos extensions of boundary functions in $\BMO$, and conjectured that the converse should also hold. It was shown by the second named author and Zacharopoulos \cite{mourgoglou2023varopoulos} that in corkscrew domains with Ahlfors regular boundaries, a regularized version of the standard dyadic extension is $L^p$- and uniformly $\ve$-approximable, which, under mild connectivity assumptions near the boundary, allowed them to construct $L^p$ and $\BMO$ Varopoulos extensions. They also proved a higher co-dimensional version of these results. In particular, this showed that uniform rectifiability of the boundary is not necessary for such constructions. 

Finally, we turn to our theorem. There we show that the trace map is surjective onto the \textit{Sobolev space} $M^{1,p}(\partial\Omega)$, from the space of functions $u$ on $\Omega$ for which $\|\mathcal{N}(\nabla u)\|_p$ and the non-tangential square function of the Hessian of $u$ are finite. Note that we do not work in $\mathbb{R}^{d+1}_+$, but rather in the more general case of a corkscrew domain with UR boundary. A similar extension was constructed by the second author and Tolsa in \cite{mourgoglou2021regularity} as well as by the second author, Poggi, and Tolsa \cite{mourgogloupoggitolsa} to solve the regularity problem for the Laplacian and for elliptic operators satisfying the so-called Dahlberg-Kenig-Pipher condition respectively. We also remark that the auxiliary extension in \cite{mourgoglou2021regularity} and  was in fact borrowed from the current work.

\section{Notation and preliminaries}\label{s:notation}
We write $a \lesssim b$ if there exists a constant $C$ such that $a \leq Cb$. By $a \sim b$ we mean $a \lesssim b \lesssim a$.
In general, we will use $n\in \mathbb{N}$ to denote the dimension of the ambient space $\R^n$, while we will use $d \in \mathbb{N}$, with $d\leq n-1$, to denote the dimension of a subset $E \subset \R^n$.
For two subsets $A,B \subset \R^n$, we let
$
\dist(A,B) := \inf_{a\in A, b \in B} |a-b|.
$
For a point $x \in \R^n$ and a subset $A \subset \R^n$, 
$
\dist(x, A):= \dist(\{x\}, A)= \inf_{a\in A} |x-a|.
$
We write 
$
B(x, r) := \{y \in \R^n \, |\,|x-y|<r\},
$
and, for $\lambda >0$,
$
\lambda B(x,r):= B(x, \lambda r).
$
At times, we may write $\mathbb{B}$ to denote $B(0,1)$. When necessary we write $B_n(x,r)$ to distinguish a ball in $\R^n$ from one in $\R^d$, which we may denote by $B_d(x, r)$. 
We denote by $\dG(n,d)$ the Grassmannian, that is, the manifold of all $d$-dimensional linear subspaces of $\R^n$. A ball in $\dG(n,d)$ is defined with respect to the standard metric
\begin{align*}
	d_{\dG}(V, W) = \|\pi_V - \pi_W\|_{{\rm op}}.
\end{align*}
Recall that $\pi_V: \R^n \to V$ is the standard orthogonal projection onto $V$.
With $\dA(n,d)$ we denote the affine Grassmannian, the manifold of all affine $d$-planes in $\R^n$. 
\noindent
The set of all affine maps $A: \mathbb{R}^n \to \R$ will be denoted as $\dM(n,1)$. Finally, $\dH^d|_E$ denotes the $d$-dimensional Hausdorff measure restricted to $E \subset \R^n$.

\subsection{Sobolev spaces}\label{s:sobolev}
\begin{definition}\label{def-Hajlasz-grad}
	Let $(X,\mu)$ be a metric measure space. For $1\leq p <\infty$, we let $M^{1,p}(X)$ the set of functions $u \in L^{p}(X)$ for which there exists a $g \in L^p(X)$ so that
	\begin{align}\label{e:upper-gradients}
		|u(x)-u(y)| \leq |x-y| (g(x)+g(y)) \mbox{ for } \mu\mbox{-a.e.} \, x,y \in X.
	\end{align}
	For $f \in L^{p}(X)$, denote by $\Grad_p(f)$ the set of $L^p(X)$ functions $g$ which satisfy \eqref{e:upper-gradients}. We also denote by $|\grad_H f|$ the function $\in \Grad_p(f)$ so that 
	\begin{equation}\label{e:Hajlasz-grad}
		\|\grad_H f\|_{L^p(X)} = \inf_{g \in \Grad_p(f)} \|g\|_{L^p(X)}.
	\end{equation}
	We call $\grad_H f$ the Haj\l{}asz gradient. If $g \in \Grad_p(f)$, we will refer to it as a \textit{Haj\l{}asz upper gradient}.
\end{definition}
\noindent We refer the reader to \cite[Section 5.4]{heinonen2005lectures} for an introduction to Haj\l{}asz-Sobolev spaces. A very useful fact about $M^{1,p}(X)$ is that pairs $(f, g)$, where $f \in M^{1,p}(X)$ and $g \in \Grad_p(X)$, always admit a Poincar\'{e} inequality.
\begin{proposition}\label{t:Hajlasz-poincare}
	Let $(X,\mu)$ be a metric measure space. Let $1\leq p <\infty$, $f \in M^{1,p}(X)$ and $g \in \Grad_p(f)$. Then for each $1\leq p' \leq p$, 
	\begin{equation}
		\left(\avint_B |f-f_B|^{p'} d\mu \right)^{\frac{1}{p'}} \leq 2 r_B \left( \avint_B g^{p'} d\mu \right)^{\frac{1}{p'}}.
	\end{equation}
\end{proposition}
\noindent
See \cite[Theorem 5.15]{heinonen2005lectures} or \cite[Proposition 2.1]{mourgoglou2021regularity} for a proof.\\

\noindent
Haj\l{}asz upper gradients should not be confused with what are commonly referred to simply as \textit{upper gradients}.
\begin{definition}\label{d:upper-grad}
	Given a metric measure space $X$ and a function $f: X \to \R$ measurable, we say that $\rho: X \to [0, \infty]$ is an upper gradient of $f$ if, for $x,y\in X$,
	$
	|u(x)-u(y)|\leq \int_{\gamma} \rho
	$
	for any rectifiable curve $\gamma$ connecting $x$ to $y$ in $X$.
\end{definition}
\noindent
Now, if the space $X$ is so that a Poincar\'{e} holds for $f$ and \textit{all of its upper gradients} (something that comes for free when using Haj\l{}asz upper gradients), then we say that $X$ admits a Poincar\'{e} inequality. More precisely:
\begin{definition}\label{def-poincare}
	For $p\geq 1$, a metric measure space $(X, d,\mu)$ admits a  {\it weak $(1,p)$-Poincar\'{e} inequality} for all measurable functions $f$ with constants $C_1,\Lambda \geq 1$ if $\mu$ is locally finite and
	\begin{equation}
		\label{e:poincare}
		\avint_{B} |f-f_{B}|d\mu \leq C_1 r_B \ps{\avint_{\Lambda B} \rho^{p}d\mu}^{\frac{1}{p}}
	\end{equation}
	where $\rho$ is any upper gradient for $f$. 
\end{definition}
\noindent

	\subsection{Dyadic lattices}\label{s:cubes}
	
	Given an Ahlfors $d$-regular measure $\mu$ in $\R^{n}$, we consider 
	the dyadic lattice of ``cubes'' built by David and Semmes in \cite[Chapter 3 of Part I]{david-semmes93}. The properties satisfied by $\DD_\mu$ are the following. 
	Assume first, for simplicity, that $\diam(\supp\mu)=\infty$). Then for each $j\in\mathbb{Z}$ there exists a family $\DD_{\mu,j}$ of Borel subsets of $\supp\mu$ (the dyadic cubes of the $j$-th generation) such that:
	\begin{itemize}
		\item[$(a)$] each $\DD_{\mu,j}$ is a partition of $\supp\mu$, i.e.\ $\supp\mu=\bigcup_{Q\in \DD_{\mu,j}} Q$ and $Q\cap Q'=\varnothing$ whenever $Q,Q'\in\DD_{\mu,j}$ and
		$Q\neq Q'$;
		\item[$(b)$] if $Q\in\DD_{\mu,j}$ and $Q'\in\DD_{\mu,k}$ with $k\leq j$, then either $Q\subset Q'$ or $Q\cap Q'=\varnothing$;
		\item[$(c)$] for all $j\in\mathbb{Z}$ and $Q\in\DD_{\mu,j}$, we have $2^{-j}\lesssim\diam(Q)\leq2^{-j}$ and $\mu(Q)\approx 2^{-jd}$;
		\item[$(d)$] there exists $C>0$ such that, for all $j\in\mathbb{Z}$, $Q\in\DD_{\mu,j}$, and $0<\tau<1$,
		\begin{equation}\label{small boundary condition}
			\begin{split}
				\mu\big(\{x\in Q:\, &\dist(x,\supp\mu\setminus Q)\leq\tau2^{-j}\}\big)\\&+\mu\big(\{x\in \supp\mu\setminus Q:\, \dist(x,Q)\leq\tau2^{-j}\}\big)\leq C\tau^{1/C}2^{-jd}.
			\end{split}
		\end{equation}
		This property is usually called the {\em small boundaries condition}.
		From (\ref{small boundary condition}), it follows that there is a point $x_Q\in Q$ (the center of $Q$) such that $\dist(x_Q,\supp\mu\setminus Q)\gtrsim 2^{-j}$ (see \cite[Lemma 3.5 of Part I]{david-semmes93}).
	\end{itemize}
	We set $\DD_\mu:=\bigcup_{j\in\mathbb{Z}}\DD_{\mu,j}$. \\
	
	\noindent
	In case that $\diam(\supp\mu)<\infty$, the families $\DD_{\mu,j}$ are only defined for $j\geq j_0$, with
	$2^{-j_0}\approx \diam(\supp\mu)$, and the same properties above hold for $\DD_\mu:=\bigcup_{j\geq j_0}\DD_{\mu,j}$.
	Given a cube $Q\in\DD_{\mu,j}$, we say that its side length is $2^{-j}$, and we denote it by $\ell(Q)$. Notice that $\diam(Q)\leq\ell(Q)$. 
	We also denote 
	\begin{equation}\label{defbq}
		B(Q):=B(x_Q,c_1\ell(Q)),\qquad B_Q = B(x_Q,\ell(Q)),
	\end{equation}
	where $c_1>0$ is some fix constant so that $B(Q)\cap\supp\mu\subset Q$, for all $Q\in\DD_\mu$. Clearly, we have $Q\subset B_Q$.
	For $\lambda>1$, we write
	$$\lambda Q = \bigl\{x\in \supp\mu:\, \dist(x,Q)\leq (\lambda-1)\,\ell(Q)\bigr\}.$$
	
	\noindent
	The side length of a ``true cube'' $P\subset\R^{n}$ is also denoted by $\ell(P)$. On the other hand, given a ball $B\subset\R^{n}$, its radius is denoted by $r_B$ or $r(B)$. For $\lambda>0$, the ball $\lambda B$ is the ball concentric with $B$ with radius $\lambda\,r(B)$.
	
	\vvv

	\subsection{Uniform rectifiability}\label{s:alpha}
	 
	\begin{definition}[Uniform rectifiability]
		We say that an Ahlfors $d$-regular set $E \subset \R^n$ is uniformly $d$-rectifiable if it contains "big pieces of Lipschitz images" (BPLI) of $\R^d$. That is to say, if there exist constants $\theta, L>0$ so that for every $x \in E$, and $0<r<\diam(E)$, there is a Lipschitz map $\rho: \R^d \to \R^n$ (depending on $x,r$), with Lipschitz constant $\leq L$, such that
		\begin{equation*}
			\dH^d\left(E \cap B(x,r) \cap \rho(B(0, r))\right) \geq \theta r^d.
		\end{equation*}
	\end{definition}
	We might often simply say uniformly rectifiable or UR sets. There is a well developed theory of uniformly rectifiable sets. We refer the interested reader to the original monographs \cite{david-semmes91} and \cite{david-semmes93}.
	We report some well-known geometric facts about uniformly rectifiable sets. They will come in handy later on.  
	\begin{lemma}\label{l:bilateral-corona}
		Let $\pom \subset \R^{d+1}$ be uniformly $d$-rectifiable. Let $0< \eta \ll \alpha \ll 1$ be sufficiently small, and $A \geq 1$ sufficiently large. The following holds: for $\sigma$-almost all $\xi \in \pom$, there is a cube $R_\xi$ and an $\alpha$-Lipschitz graph  $\Gamma_\xi$, so that 
		\begin{enumerate}
			\item $\xi \in \Gamma_\xi$.
			\item For all $Q$ containing $\xi$ and with $\ell(Q)\leq \ell(R_\xi)$ we have
			\begin{equation}\label{e:bilateral-beta-small}
				b\beta_{\pom, \infty}(A Q):= \inf_P d_H(A B_Q \cap \pom, AB_Q \cap L) < \eta \ell(Q),
			\end{equation}
			where the infimum is over all affine $d$-planes $P$. 
			\item Denoting by $P_Q$ an infimising plane in \eqref{e:bilateral-beta-small}, then $\angle(P_{R_\xi}, P_Q)< \alpha$ for all $Q \ni \xi$, $Q\subset R_\xi$. 
			\item For all $Q$ containing $\xi$ and with $\ell(Q)\leq \ell(R_\xi)$ we have
			\begin{equation}\label{e:coronisation-bilateral}
				d_H(AB_Q \cap \pom, AB_Q \cap \Gamma_\xi) < \eta\ell(Q)
			\end{equation}
			\item We have that 
			\begin{gather}
				\Gamma_\xi \cap AB_{R_\xi} \subset  X(\xi, P_{R_\xi}, 2\alpha, A \ell(R_\xi)). \\
				\pom \cap AB_{R_\xi} \subset X(\xi, P_{R_\xi}, 2\alpha, A \ell(R_\xi)).
			\end{gather}
		\end{enumerate}
	\end{lemma}
	\noindent
	Set
	\begin{equation*}
		\bwgl(Q_0) :=\left\{ Q \in \dD_\sigma(Q_0) \, |\, b\beta_{\pom, \infty}(A Q):= \frac{\inf_L d_H(A B_Q \cap \pom, AB_Q \cap L)}{\ell(Q)} > \eta\right\},
	\end{equation*}
	and
	\begin{equation}
		\dG(Q_0) :=\dD_\sigma(Q_0) \setminus \bwgl \quad \mbox{ and } \dG_j(Q_0) := \dD_{\sigma, j} (Q_0) \setminus \bwgl(Q_0).
	\end{equation}
	We will often simply write $\dG_j$ in stead of $\dG_j(Q_0)$.

	\subsubsection{Tolsa's $\alpha$ numbers}
	We first define Tolsa's $\alpha$ numbers. They first appeared in the area in \cite{tolsa2009uniform} in connection to singular integral operators, and have been heavily used since then.
	Let $\mu$ and $\nu$ be Radon measures. For an open ball $B$ define
	\begin{equation*}
		F_B(\sigma,\nu):= \sup\Bigl\{ \Bigl|{\textstyle \int \phi \,d\sigma  -
			\int \phi\,d\nu}\Bigr|:\, \phi\in \Lip(B) \Bigr\},
	\end{equation*}
	where 
	\begin{equation*}
		\Lip(B)=\{\phi:{\rm Lip}(\phi) \leq1,\,\supp f\subseteq
		B\}
	\end{equation*}
	and $\Lip(\phi)$ stands for the Lipschitz constant of $\phi$. See \cite[Chapter 14]{Mattila} for the properties of this distance. Next, set
	\begin{gather}
		\alpha_\sigma^{d}(B,P) := \frac1{r_{B}\,\sigma(2B)}\,\inf_{c\geq0} \,F_{2B}(\sigma,\,c\dH^{d}|_{P}),\label{e:alpha-def0}\\
		\alpha_{\sigma}^{d}(B) = \inf_{P\in \dA(d,n)}\alpha_\sigma^{d}(B,P)\label{e:alpha-def}.
	\end{gather}
	Note that the right hand side of \eqref{e:alpha-def} is computed over $2B$ (rather than $B$). This is simply for notational convenience. 
	
	\begin{remark}
		We denote by $c_B$ and $P_B$ a constant and a plane that infimise $\alpha_\sigma(B)$. That is, we let $c_B>0$ and $P_B \in \dA(n,d)$ be such that, if we set
		\begin{equation}\label{e:def-L}
			\mathcal{L}_{B} :=c_B\dH^{d}|_{P_B}, 
		\end{equation}
		then
		\begin{equation}\label{cL-ppts}
			\alpha_\sigma^{d}(B) =  
			\alpha_{\sigma}^{d}(B,\mathcal{L}_{B})=\frac1{r_{B}^{d+1}}\,F_{2B}(\sigma,\,\mathcal{L}_B)
		\end{equation}
	\end{remark}

	\subsubsection{Jones' $\beta$ numbers}
	The second quantity we introduce are the well-known Jones' $\beta$ numbers. For a ball $B$ centered on $\Sigma$, a $d$-plane $P \in \dA(n,d)$, and $p>0$, put
	\begin{equation*}
		\beta_{\sigma}^{d,p}(B,P)  =  \left(\frac{1}{r_B^{d}} \int_{B}\ps{\frac{\dist(y,L)}{r_B}}^{p} d\sigma(y)\right)^{\frac{1}{p}}.
	\end{equation*}
	The Jones' $\beta$-number of $\Sigma$ in the ball $B$ is defined as the infimum over all $d$-affine planes $P \in \dA(n,d)$:
	\begin{equation*}
		\beta_{\sigma}^{d,p}(B) =\inf_{P\in \dA(d,n)} \beta_{\sigma}^{d,p}(B,P).
	\end{equation*}

\subsection{The coefficients $\Omega$ and $\gamma$} In this subsection we introduce the quantities relevant to Dorronsoro's estimates.
Let $q \geq 1$ and consider a function $f: \Sigma \to \R$ so that $f\in L^{q}(\Sigma)$. For a each ball $B$ centered on $\Sigma$, and an affine map $A : \R^n \to \R$, let 
\begin{equation}\label{e:def-omega}
	\Omega_{f}^{q}(B,A)=\ps{ \avint_{B}\ps{\frac{|f-A|}{r_B}}^{q}d\sigma}^{\frac{1}{q}}
	\quad \mbox{ and } \quad  \Omega_{f}^{q}(B)=\inf_{A \in \dM(n, 1)}\Omega_{f}^{q}(B,A).
\end{equation}

We now come to the definition of the quantity $\gamma_f$. Let $f$ be a real valued function defined on $E \subset \R^n$.
\begin{definition}\label{d:gamma}\
	\begin{itemize}
		\item For $1\leq q \leq \infty$, $f \in L^q(\Sigma)$ and $A \in \dM(n,1)$, set
		\begin{equation}\label{e:def-pgamma0}
			\gamma_{f}^{q}(B,A) = \Omega_{f}^{q}(B,A)+|\grad A|\beta_{\Sigma}^{d,q}(B). 
		\end{equation}
		Then set
		\begin{equation}\label{e:def-pgamma}
			\gamma_{f}^{q}(B)=\inf_{A \in \dM(n,1)} \gamma_{f}^{q}(B,A).
		\end{equation}
		\item If $f \in L^1(\Sigma)$ and $A \in \dM(n,1)$ let 
		\begin{equation}\label{e:def-1gamma0}
			\widetilde{\gamma}_{f}(B,A)= \Omega_{f}^{1}(B,A)+|\grad A|\salpha(B),
		\end{equation}
		and then
		\begin{equation}\label{e:def-1gamma}
			\widetilde{\gamma}_{f}(B)=\inf_{A\in \dM(n,1)} \widetilde\gamma_{f}(B,A).
		\end{equation}
	\end{itemize}
	It is immediate from the definitions that for $q\geq 1$, $\gamma_f^q(B) \geq \Omega_f^q(B)$.
\end{definition}
We refer to Section 2 of \cite{AMV1} for properties of these coefficients. The following lemma will be useful. Its proof may be found in \cite{AMV1}. If $B$ is a ball, $L_B$ the plane minimising either $\alpha(B)$ of $\beta(B)$, then $\pi_B$ denotes the orthogonal projection onto $L_B$.

\begin{lemma}
	\label{l:A_B}
	Let $\Sigma \subset \R^n$ be a Ahlfors $d$-regular subset, $B$ a ball centered on $\Sigma$ and $q\geq 1$. Let $f \in L^q(\Sigma)$. There is an affine map in $\dM(n,1)$, denoted by $A_{B}$, so that 
	\begin{align}
		\label{e:ABbeta<gammaAB<gamma}
		|\grad A_{B}|\beta_{\Sigma}^{d,q}(B)\lec \gamma_{f}^{q}(B,A_B) 
		\lec \gamma_{f}^{q}(B), \\
		|\grad A_{B}|\beta_{\Sigma}^{d,1}(B)\lec \widetilde{\gamma}_{f}(B,A_B) 
		\lec \widetilde{\gamma}_{f}(B), \label{e:ABbeta<gammatilde}
	\end{align}
	and
	\begin{equation}
		\label{e:ABpiAB}
		A_{B}\circ\pi_{B}=A_{B}.
	\end{equation}
\end{lemma}

\subsection{Whitney regions, well connected components, good corkscrews}\label{ss:whitney}

Let $\om \subset \R^{n+1}$ be an open set.	We 
consider the following Whitney decomposition of $\Omega$ (assuming $\Omega\neq \R^{n+1}$): we have a family $\WW(\Omega)$ of dyadic cubes in $\R^n$ with disjoint interiors such that
$$\bigcup_{I\in\WW(\Omega)} I = \Omega,$$
and moreover there are
some constants $\Lambda>20$ and $D_0\geq1$ such the following holds for every $I\in\WW(\Omega)$:
\begin{enumerate}
\item $10\cdot I \subset \Omega$;
\item $\Lambda \cdot I \cap \partial\Omega \neq \varnothing$;
\item there are at most $D_0$ cubes $P'\in\WW(\Omega)$
such that $10I \cap 10J \neq \varnothing$. Further, for such cubes $J$, we have $\frac12\ell(J)\leq \ell(I)\leq 2\ell(J)$.
\end{enumerate}
From the properties (1) and (2) it is clear that 
\begin{equation}\label{eqeq29}
10 \ell(I) \leq \dist(I,\partial\Omega) \leq \Lambda \ell(I).
\end{equation}
The arguments to construct a Whitney decomposition satisfying the properties above are
standard.
\begin{remark}
In general, we will denote Whitney cubes by $I, J \in \dW(\om)$ and Christ-David cubes by $Q,P,R \in \dD_{\sigma}$.
\end{remark}

\noindent
Let $I \in \WW(\Omega)$. Let $\xi\in \pom$ be a closest point to $I$, that is, a point satisfying $\dist(\pom, I) \leq \dist(\xi, I) \leq 2 \dist(\pom, I)$. Then $\dist(\xi, I) \approx_\Lambda \ell(I)$. But also $\xi \in Q$ for some cube with $\ell(I) = \ell(Q)$. This cube, which we denote by $Q_I$, will have the property that 
\begin{equation}\label{e:Q_I}
\ell(Q_I) = \ell(I) \mbox{ and } 10 \ell(I) \leq \dist(Q_I,P) \leq \Lambda \ell(I).
\end{equation}
Conversely, for some $0<\tau<1$, and
given $Q\in\DD_\sigma$, we let 
\begin{align}
\widetilde{\dW}_\tau (Q):= \{ P \in \WW(\Omega) \, |\, & \tau \cdot  \ell(Q) \leq \ell(P) \leq \ell(Q)  \label{e:w'(Q)} \,\\
& \mbox{ and }  \tau  \cdot \ell(Q) \leq \dist(Q, P) \leq \Lambda \ell(Q) \}, \label{e:w'Q2}
\end{align}
and
\begin{equation}\label{eqwq00}
\widetilde{W}_\tau(Q) = \bigcup_{P\in \dW_c(Q)} P.
\end{equation}
We will most often suppress the dependence on $\tau$ in the notation and simply write $\widetilde{\dW}(Q)$ or $\widetilde{W}(Q)$. 
\begin{definition}[\textit{Well-connected components}]\label{remark:components}
Note that $\widetilde{W}(Q)$ might consist of more than just one (quantitatively) connected component. More precisely, we decompose $\widetilde{W}(Q)$ as follows: we say that a subset $v_Q \subset \widetilde{W}(Q)$ is a \textit{well-connected component} if any two points $x,y \in v_Q$ can be joined by a $\theta$-cigar curve. Denote by $M_Q$ the number of distinct well-connected components of $\widetilde{W}(Q)$. We will use the notation $\{v^i_Q\}_{i=1}^{M_Q}$. We might also abuse notation and denote by $\dW(Q)=\{v_Q^i\}_{i=1}^{M_Q}$.
\end{definition}

\begin{remark}\label{r:at-most-2}
Note also that if $Q \in \dD_{\sigma} \setminus \bwgl$, then there are at most two well-connected components $v_Q^1, v_Q^2 \subset \widetilde{W}(Q)$. Indeed, if $P_Q$ is the plane minimising $b \beta_{\pom, \infty}(A Q)$, then $\widetilde{W}(Q) \subset B_Q \setminus P_Q(\eta \ell(Q))$. Denoting by $D^+$ and $D^-$ the connected components of $B_Q \setminus P_Q(\eta \ell(Q))$, we see that if $v_Q^1, v_Q^2 \subset D^+$, say, are two well-connected components of $\widetilde{W}(Q)$, then we might join each pair $x_1 \in v_Q^1, x_2\in v_Q^2$ with a $\theta$-cigar. Hence $v_Q^1 = v_Q^2$. Thus, if $Q \notin \bwgl$, we might have at most two well-connected components. 
\end{remark}

\begin{lemma}\label{l:ball-cube}
Let $\Omega \subset \R^{d+1}$ be a local John domain with constant $\theta$. 
\begin{itemize}
	\item Choosing $0<\tau<1$ in \eqref{eqwq00} appropriately depending on $\theta$, for each $Q \in \dD_\sigma$, we can ensure that $\widetilde{W}(Q)=\widetilde{W}_\tau(Q) \neq \varnothing$ and that there is at least one cube $I \in \widetilde{\dW}(Q)$ so that $I \subset B^Q$. 
	\item Furthermore, to any corkscrew ball $B^Q \subset B_Q\cap \om$ of radius at least $\theta \ell(Q)$, there is a well connected component $v_Q$ such that $0.5B^Q \subset v_Q$. 
\end{itemize}
\end{lemma}
\begin{proof}
This is immediate from definitions: let $\tau=\theta/10$; if $Q \in \dD_\sigma$ then we know that there is a corkscrew ball $B^Q \subset B_Q \cap \om$ with radius $\geq \theta \ell(Q)$. Then $x_{B^Q}$ will be contained in a dyadic cube $I$ with sidelength $\sim \theta/5$. Then $I$ will satisfy both \eqref{e:w'(Q)} and \eqref{e:w'Q2}. The second part of the lemma is similarly immediate.
\end{proof}

\begin{definition}[\textit{Good corkscrews}]
Given $Q \in \dD_\sigma$, we say that a corkscrew ball $B^Q \subset B_Q \cap \om$ with radius $\geq \theta \ell(Q)$ is a \textit{good corkscrew ball} if for each $\xi \in \pom\cap B_Q$, there is a $\theta$-carrot path connecting $x_{B^Q}$ to $\xi$. If $\om$ is a local John domain, we are always guaranteed the presence of at least one good corkscrew ball.
\end{definition}

\begin{definition}[\textit{Good well-connected component}]\label{def:good-well-connected-components}
Let $Q \in \dD_\sigma$ and let $B^Q$ be a good corkscrew ball. We call the well connected component $v_Q \subset \widetilde{W}(Q)$ which contains $0.5B^Q$ a \textit{good} well connected component. Of course, there might be multiple good well connected component for one $Q \in \dD_\sigma$. Abusing notation, for each $Q$, we denote both the union and the family of good well connected components by $W(Q)$; the family of Whitney cubes constituting good well connected components will be denoted by $\dW(Q)$.
\end{definition}

\subsection{Compatible choices}
What we would like to do now, is to set
\begin{equation}\label{e:uQ}
u_Q := \avint_{B^Q} u(x)\, dx,
\end{equation}
and then define $f$ as the limit of these averages when $\ell(Q) \to 0$. However, care must be taken when choosing \textit{which good corkscrew ball} to use when taking the average. We need some compatibility in this choice. Lemma \ref{l:compatibility} below gives us this correct choice.

\begin{lemma}\label{l:compatibility}
For each $j \geq 1$, there exists a choice of good corkscrew balls $\{B^Q\}_{Q \in \dD_{\sigma, j}}$ such that, for each pair $Q,Q' \in \dG_j$, if $R \in \dD_{\sigma, i}$, $i\leq j$, is the minimal cube satisfying $\xi_{Q},\xi_{Q'} \in 3B_R$, then we can find carrot paths joining $x_{B^Q}, x_{B^{Q'}}$ to the center of a (common) good corkscrew ball in $3B_R$ of radius $3 \theta \ell(R)$.
\end{lemma}
\begin{remark}\label{r:compatible-choice}
We call a family of good corkscrew ball $\{B^Q\}_{Q \in \dG_j}$ which satisfy the conclusions of Lemma \ref{l:compatibility} a \textit{compatible choice} of corkscrew balls. Of course, for each $Q$ there might be (uncountably) many corkscrew balls that form a compatible choice. 
\end{remark}
\begin{proof}[Proof of Lemma \ref{l:compatibility}]
Suppose the lemma is false. Then for some $j \geq 1$, we can find a pair $Q,Q' \in \dD_{\sigma, j}$, such that no choice of good corkscrew balls $B^Q, B^{Q'}$ can be made so that both $B^Q$ and $B^{Q'}$ are connected via carrot path to one good corkscrew ball in $3B_R$ of radius at least $3 \theta \ell(R)$. But consider $\xi_R$, the center of $R$. Then, by the local John condition, we can find a good corkscrew ball, denoted by $B^{3R}$, with  $B^{3R} \subset B(\xi_R, 3\ell(R)) \cap \Omega$, and with radius $r(B^{3R})=3\theta \ell(R)$. In particular, there is a $\theta$-carrot path connecting $x_{B^{3R}}$ (the center of $B^{3R}$) to both $\xi_Q$ and $\xi_{Q'}$. Each of these carrot paths provide us with a corkscrew balls $c(Q) \subset B(\xi_Q, \ell(Q)) \cap \Omega$, $c(Q') \subset B(\xi_{Q'}, \ell(Q)) \cap \om$ of radius $\theta \ell(Q)$. It is not necessarily true, however, that both $c(Q)$ and $c(Q')$ are \textit{good} corkscrew balls. If both of them are, then we are done. We denote by $D^+, D^-$ the two connected component of $AB_Q \setminus P_Q(\eta \ell(Q))$. Note that since $Q \in \bwgl(Q_0)$, $D^+ \cup D^- \cap \pom = \varnothing$. Assume without loss of generality that $c(Q) \subset D^+$ and suppose $c(Q)$ is not a good corkscrew ball. This means that there is a point $ \zeta_Q \in B(\xi_Q, \ell(Q))\cap \pom$ which cannot be connected to $x_{C(Q)}$, the center of $C(Q)$, by a carrot path. Note that, then, $\zeta_Q$ cannot be connected to $x_{B^{3R}}$ either, at least via a carrot path that passes through $D^+$. Indeed, suppose that this could be done: we can find a carrot path joining $x_{B^{3R}}$ to $\zeta_Q$. But then there is a corkscrew ball, denoted by $B_{\zeta_Q}$ and of radius $\theta \ell(Q)$, which is contained in $2B_Q\cap \om \cap D^+$ and whose center is connected with a carrot path to $\zeta_Q$. Hence both $B_{\zeta_Q} \subset D^+$ ans $C(Q) \subset D^+$. But this implies that we can join $B_{\zeta_Q}$ and $C(Q)$ with a cigar curve, and thus construct a carrot path joining $x_{C(Q)}$ to $\zeta_Q$, and this lead to a contradiction. Hence there cannot be a carrot path passing through $D^+$ which joins $x_{B^{3R}}$ to $\zeta_Q$. But since $B^{3R}$ is a good corkscrew ball, we conclude that there is a carrot path joining $x_{B^{3R}}$ to $\zeta_R$ passing through $D^-$, as there is no other option. Note at this point that there must be a good corkscrew ball $C'(Q) \subset D^-$, by the local John condition. This ball can be joined with a cigar curve to the corkscrew ball contained in the carrot path joining $x_{B^{3R}}$ to $\zeta_Q$, by the fact that $b\beta_{\pom, \infty}(AQ)<\eta$. We conclude that we can join the good corkscrew ball $C'(Q)$ to $B^{3R}$ with a carrot path. We now repeat the same argument for $Q'$ and conclude that we can choose two good corkscrew balls $B^Q$ and $B^{Q'}$ to a common good corkscrew ball $B^{3R}$ via carrot path. This contradicts the assumption made at the beginning of the proof, and the lemma is proven.
\end{proof}

\begin{remark}\label{r:dB(Q)}
Given $Q \in \dG_j$, let $\wt{\dB}(Q)$ be the family of corkscrew balls that belong to at least a compatible choice (as in Remark \ref{r:compatible-choice}). For each good well connected component $v_Q$, we pick one good corkscrew ball $B^Q \in \wt{\dB}(Q)$ (where the relation between $v_Q$ and $B^Q$ is as in Lemma \ref{l:ball-cube}). We call this subfamily $\dB(Q)=\{B^Q_i\}_{i=1}^{N_Q}$. Note that $N_Q:=\#\dB(Q) \leq 2$ by Remark \ref{r:at-most-2} For each $Q \in \dG$, we also define the family of good well connected components which contains (half of) a corkscrew ball from $\dB(Q)$. That is

\begin{equation}\label{e:compatible-whitney}
	\cw(Q) := \{ v_Q \, |\, v_Q \in W(Q) \mbox{ and there is } B^Q \in \dB(Q) \mbox{  s.t. } 0.5B^Q \subset v_Q\}.
\end{equation}
\end{remark}
\noindent

\subsection{Non-tangential cones and the spaces $\dV_{p, \lambda}$; discrete non-tangential cones and the space $\bV_{p, \lambda}$; truncated non-tangential cones}

For technical reasons, we introduce some variants of the non-tangential cone and of the spaces defined in the introduction and recalled at the beginning og this Part III. For a parameter $\lambda>0$ and $\xi \in \pom$, set
\begin{equation*}
\Gamma^\lambda(\xi)= \left\{  y \in \Omega \, |\,  \lambda \cdot |\xi-y| < d_\Omega(y)  \right\}.
\end{equation*}
\noindent
Before, we had that $\Gamma(\xi)= \Gamma^{1/2}(\xi)$.
For a vector field $F: \om \to \R^m$, $1 \leq m<+\infty$, define the corresponding non-tangential maximal function as 
$
\dN^\lambda (F)(\xi)= \sup_{y \in \Gamma^\lambda(\xi) } |F(y)|.
$
Then set 
\begin{equation}\label{e:Vp1/2}
\dV_{p, \lambda}(\om) := \left\{u \in L^p(\om) \, |\, \dN^\lambda(\grad u) \in L^p(\pom)\right\},
\end{equation}
and
\begin{equation*}
\dS_\lambda (u)(x)= \left( \int_{\Gamma^\lambda(x)} |u(y)|^2\,\frac{dy}{d_\om(y)^{d+1}} \right)^{1/2},
\end{equation*}
together with the corresponding space
\begin{equation*}
\dV_{p, \lambda}^S (\om) := \left\{ u \in C^2(\om) \, |\, \dN^\lambda(\grad u) \in L^p(\pom)\, \mbox{ and }\, \dS_\lambda(d_\om(\cdot) \nabla^2 u) \in L^p(\pom)\right\}.
\end{equation*}
\begin{lemma}\label{l:non-tan-lambda}
For any $\lambda>0$, $\dV_{p, \lambda} = \dV_{p, 1/2}= \dV_{p}$. 
\end{lemma}
\begin{proof}
That $\dV_{p} = \dV_{p, \frac{1}{2}}$ is by definition (recall \eqref{e:Vp1/2}). That $\dV_{p, \alpha} = \dV_{p, \tfrac{1}{2}}$ follows from the fact that $\|N_\alpha(\grad u)\|_p \approx_{\alpha, \beta} \|N_\beta(\grad u)\|_p$ for any $\alpha, \beta>0$. This follows using the Alhfors regularity of the boundary $\pom$ and the classical proof for $\R^{d+1}_+$ by Fefferman and Stein in \cite[Lemma 1]{fefferman1972h} (but see also \cite[Lemma 1.27]{hofmann2020uniform}).
\end{proof}

\subsubsection{Discrete versions} We will also need a "discretised" version of the non-tangential regions $\Gamma^\lambda(\xi)$; the definition is as in \eqref{e:def-Gaw}, except that we take the union over all cubes $Q \in \dD_\sigma$ containing $\xi$. More precisely, we set
\begin{equation}\label{e:discrete-non-tan}
\mathbb{Y}_{{\text{\tiny{W}}}}^\lambda(\xi) := \bigcup_{Q \in \dD_\sigma \atop Q \ni \xi} \bigcup_{\substack{P \in \dD_\sigma \\  \ell(P)= \ell(Q) \\ \lambda P \ni \xi}} \wt{W}(Q). 
\end{equation}
Remark that this definition depends on $\lambda>0$, but also on $\tau>0$ in the definition of $\wt{W}(Q)$ in \eqref{eqwq00}. Define the corresponding maximal function
\begin{equation*}
N_{{\text{\tiny{W}}}}^\lambda (F)(\xi) := \sup_{y \in \Omega \cap \bY^\lambda(y)} |F(y)|.
\end{equation*}
and the corresponding space
\begin{equation*}
\bV_{p, \lambda}:= \mathbb{V}_{p, \lambda} (\om) := \left\{ u \in L^p(\om) \, |\, \Nw^\lambda (\grad u) \in L^p(\pom)\right\}.
\end{equation*}
\noindent Note in passing, that if $\lambda_1 \leq \lambda_2$, then $\Nw^{\lambda_1} (F)(\xi) \leq \Nw^{\lambda_2} (F)(\xi)$. This fact will be used below without explicit mention.
\begin{lemma}\label{l:non-tan-lambda-discrete}
For any $\lambda>0$, $\dV_{p, \lambda} = \bV_{p, \lambda}$.
\end{lemma}
\noindent
The proof is as in \cite{hofmann2020uniform}, Lemma 1.27. The lemma there is stated for continuous functions $u$, but it in fact holds more generally for measurable functions; see, for example, \cite{mourgoglou2021regularity}, equations (1.7) and (1.8). Remark also that in \cite{hofmann2020uniform} the authors use a different but equivalent definition of "discretised" non-tangential region; compare Definition 1.25 in \cite{hofmann2020uniform} to our own \eqref{e:discrete-non-tan}. 

\begin{remark}
Lemmas \ref{l:non-tan-lambda} and \ref{l:non-tan-lambda-discrete} together, tell us that to prove Theorem \ref{t:extension-in-text}(1) and (2), it suffices to prove it for $u \in \bV^\lambda$, for some $\lambda >0$. 
\end{remark}

\subsubsection{Truncated versions} 
Given $\xi \in \pom$ satisfying the conclusion of Lemma \ref{l:bilateral-corona} and $R_\xi$ as in its statement, and for $\lambda>0$, set
\begin{equation}\label{e:tower}
\dT_\lambda(\xi) := \{ Q \in\dD_\sigma(R_\xi) \, |\, \lambda Q \ni \xi \}. 
\end{equation}
Define also
\begin{equation}\label{e:def-Gaw}
\Gaw^\lambda(\xi) := \bigcup_{Q \in \dT_\lambda (\xi)} \wt{W}(Q).
\end{equation} 
For simplicity, we will write $\dT(\xi)=\dT_\lambda(\xi)$ and $\Gaw(\xi) = \Gaw^\lambda(\xi)$.
\begin{remark}\label{r:H+-}
Keep the hypotheses of Lemma \ref{l:bilateral-corona}. For almost all $\xi \in \pom$, we have the following construction. Denote by $H^+$ and $H^-$ the two connected components of $B(\xi, A\ell(R_\xi)) \setminus X(\xi, P_{R_\xi}, 2\alpha, A\ell(R_\xi))$. Then clearly, either $H^+ \subset \om$, or $H^-\subset \om$, or both.
If we choose $\alpha$ sufficiently small with respect to $\theta$, we see that for each $Q \in \dT(\xi)$, and for all $v_Q \in \widetilde{W}(Q)$, $v_Q \subset H^+ \cup H^-$. We set
\begin{equation}
	\Gaw^\pm(\xi):= \Gaw(\xi) \cap H^\pm.
\end{equation}
Note that for any pair $Q,Q' \in \dT(\xi)$, if $B^Q, B^{Q'}$ are two corkscrew balls of radius $\theta \ell(Q)$ and $\theta \ell(Q')$, respectively, and they are contained in the same connected component, say $H^+$, then we may connect $x_{B^{Q'}}$ to $x_{B^Q}$ with a $\theta$-carrot path. 
\end{remark}

\section{Almost everywhere existence of the trace and pointwise Sobolev estimates}

\subsection{Definition and convergence of trace}

Fix $\lambda=3$.
\begin{remark} \label{r:lambda=3}
Because we have fixed $\lambda$, and for notational convenience, in this section we hide the dependence of $\bV_{p,\lambda}$, $\Nw^\lambda(\grad u)$, $\dT^\lambda(\xi)$ and $\bY^\lambda$ on $\lambda$, hence simply writing $\bV$, $\Nw(\grad u)$, $\dT(\xi)$ and $\bY$.
\end{remark}

\noindent	
Let $u \in \bV$. 
For $Q \in \dD_\sigma$, set 
\begin{equation*}
\Phi(Q) :=\frac{1}{N_Q} \sum_{v_Q \in \cw(Q)} u_{v_Q}
\end{equation*}
\newcommand{\N}{\mathbb{N}}
(We remind that the relevant notation was set in Remark \ref{remark:components}, Definition \ref{def:good-well-connected-components} and \eqref{e:compatible-whitney}). 

\begin{lemma}\label{l:trace-convergence}
For all those $\xi \in \pom$ satisfying the conclusions of Lemma \ref{l:bilateral-corona} (in particular, for $\sigma$-almost all $\xi \in \pom$),
\begin{equation}\label{e:limitPhiQ}
	\lim_{Q \in \dT(\xi) \,:\, \ell(Q) \to 0} \Phi(Q) \,\mbox{ exists, }
\end{equation}	
perhaps up to a subsequence. If it exists up to a subsequence, then there are exactly two subsequences, both of which converge.
\end{lemma}

\begin{remark}
Lemma \ref{l:trace-convergence} corresponds to Theorem \ref{t:extension-in-text}(2).
\end{remark}
\begin{proof}
Let $\xi$ be a point in $\pom$ so that the conclusions of Lemma \ref{l:bilateral-corona} are satisfied. 
We distinguish between two distinct cases:
\begin{enumerate}
	\item There is a $Q^* \in \dT(\xi)$ so that for all $Q \subset Q^*$ in $\dT(\xi)$, $N_Q =1$. Recalling Remark \ref{r:H+-}, within this case we consider two subcases:
	\begin{enumerate}
		\item Either there is a $Q^{**} \in \dT(\xi)$ with $Q^{**}  \subset Q^*$ so that each $Q \subset Q^{**}$ containing $\xi$ is so that $0.5B^Q \subset H^+$, where $B^Q$ is the the good corkscrew ball $B^Q \in \dB(Q)$, as chosen in Remark \ref{r:dB(Q)}; or the same happens, with $H^-$ instead of $H^+$. 
		\item No such $Q^{**}$ exists. 
	\end{enumerate}
	\item There is no such $Q^* \in \dT(\xi)$. In particular, there is a sequence of cubes $Q\in \dT(\xi)$ whose side length converges to $0$, such that $N_Q =2$.
\end{enumerate}

\noindent
\textbf{Case 1(a).} We will show that in this case, the limit in \eqref{e:limitPhiQ} exists (without having to resort to subsequences). Fix $\ve>0$, and let $P \subset Q^{**}$ be so that $\ell(P) < C \ve$, for some $C$ to be determined below. Let $S,S'$ be two cubes in $\dT(\xi)$ contained in $P$ and with $\ell(S')<\ell(S)$. Note that since we are in the current case, we have that
\begin{equation*}
	\Phi(S) = u_{v_S}, \, \mbox{ and } \, \Phi(S') = u_{v_{S'}}.  
\end{equation*}
Moreover, we can connect $x_{B^S}$ to $x_{B^{S'}}$ with a $\theta$-carrot path, which we denote by $\gamma(S,S')$. Note also that by definition, any point in $v_S$ can be connected to $x_{B^S}$ by a rectifiable path of length $\lesssim \ell(S)$. The same can be said with $S'$ replacing $S$. Denoting $z=x_{B^{S'}}$ to ease notation, we write
\begin{equation*}
	|\Phi(S) - \Phi(S')|  = |u_{v_S}- u_{v_{S'}}|
	\leq |u_{v_S} - u(z)| + |u(z) - u_{v_{S'}}| =: I_1 + I_2. 
\end{equation*}
By the paragraph above the latest display, we see that there is a rectifiable path $\gamma_{y,z}$ of length $\lesssim \ell(S)$ connecting any point $y \in v_S$ to $z= x_{B^{S'}}$. Hence we may estimate
\begin{align}
	I_1 & \leq  \avint_{v_S}| u(y) - u(z)| \, dy  \leq \avint_{v_S} \int_{\gamma(y,z)} \grad u (s) \, ds\, dy \nonumber \\
	& \leq \avint_{v_S} \int_{\gamma(y,z)} \sup_{x \in \Gaw(\xi)} \grad u(x) \, ds\, dy \lesssim \ell(S) \Nw(\grad u)(\xi).\label{e:form10000}
\end{align}
Since $\ell(S) \leq \ell(P) \leq C \ve$, choosing $C$ appropriately (depending only on $\theta$), gives $I_1 < \ve/2$. The same estimate is obtained for $I_2$ taking into account that $z \in v_{S'}$, $v_{S'} \subset \Gaw(\xi)$ and that $\ell(S') <\ell(S)$. This shows that if we are in Case 1(a), the sequence $\{\Phi(Q)\}_{Q \in \dT(\xi)}$ is a Cauchy sequence, and therefore it converges.\\

\noindent
\textbf{Case 1(b).} In this case we will show that $\{\Phi(Q)\}_{Q \in \dT(\xi)}$ consists of two converging subsequences (with possibly different limits). Recall from Remark \ref{r:H+-}, that if $Q \in \dT(\xi)$, and $B^Q$ is a corkscrew ball (good or not), then either $B^Q \subset H^+$ or $B^Q \subset H^-$. Thus if no $Q^{**} \in \dT(\xi)$ as in Case 1(a) exists, we conclude that there are two infinite families $\dT^\pm(\xi) \subset \dT(\xi)$, so that if $Q \in \dT^\pm(\xi)$ and $B^Q \in \dB(Q)$, then $B^Q \subset H^\pm$. Moreover, $\dT^+(\xi) \cup \dT^-(\xi) = \dT(\xi)$.

	\begin{claim}
		With current hypotheses and notation, we have that 
		\begin{equation}
			L^+(\xi):=\lim_{Q \in \dT^+(\xi) \, : \ell(Q) \to 0} \Phi(Q) \mbox{  exists.}
		\end{equation}
	\end{claim}
	\begin{proof}
		To prove the claim, it suffice to following the considerations that lead to the bound \eqref{e:form10000}, and recall that for any pair $Q,Q' \in \dT(\xi)$, there is a $\theta$-carrot path joining $x_{B^Q}$ to $x_{B^{Q'}}$.
	\end{proof}
	\noindent
	We similarly claim that
	\begin{claim}
		\begin{equation*}
			L^-(\xi):=\lim_{Q \in \dT^-(\xi) \,: \, \ell(Q) \to 0} \Phi(Q) \mbox{ exists.}
		\end{equation*}
	\end{claim}
	\noindent
	The proof of this is again the same as \eqref{e:form10000}, and we leave the details to the reader. Remark that the two limits have not reason to be the same.\\

	\noindent
	\textbf{Case 2.} No cube $Q^*$ as in Case 1 exists, and hence we have an infinite family of cubes $Q \in\dT(\xi)$ so that $N_Q =2$ (recall that since $Q \in \bwgl$, it is always the case that $N_Q \leq 2$ - this was clarified in Remark \ref{r:at-most-2}). Now define
	\begin{equation*}
		\cw^\pm(\xi):=\{ v_Q \in \cw(Q) \, |\, Q \in \dT(\xi) \mbox{ and } v_Q \subset H^\pm\}.
	\end{equation*}
	Recall that $\cw(Q)$ is the family of good well connected components of the Whitney region of $Q$ which contain a corkscrew ball from a compatible choice (as defined in \eqref{e:compatible-whitney}).
	
	Let $Q, Q' \in\dT(\xi)$ be so that there are $v_Q, v_{Q'} \in \cw^+(\xi)$. Then it is easily seen from Remark \ref{r:H+-} that we can connect any pair $x \in v_Q$ and $v_{Q'}$ with a $\theta$-carrot path. Hence, as in the proof of the bound \eqref{e:form10000} in Case 1(a), we can show that the limits
	\begin{equation*}
		\lim_{v_Q \in \cw^\pm(\xi) \, : \, \ell(Q) \to 0} u_{v_Q}\quad  \mbox{ exist.}
	\end{equation*}
	It can be easily checked that
	\begin{equation}\label{e:form20001}
		\lim_{Q \in \dT(\xi) \, :\, \ell(Q) \to 0} \Phi(Q) = \frac{1}{2}\left(\lim_{v_Q \in \cw^+(\xi) \, : \, \ell(Q) \to 0} u_{v_Q}+ \lim_{v_Q \in \cw^-(\xi) \, : \, \ell(Q) \to 0} u_{v_Q}\right).
	\end{equation}
\end{proof}
\noindent
For all $\xi \in \pom$ satisfying the conclusions of Lemma \ref{l:bilateral-corona}, we define the trace $f(\xi)= Tu(\xi)$ of $u$ as follows:
\begin{equation}\label{e:def-trace}
	f(\xi)= Tu(\xi) = \begin{cases}
		\lim_{Q \in \dT(\xi) \,:\, \ell(Q) \to 0} \Phi(Q) & \mbox{ if Case 1(a) or Case 2 holds.}\\
		\frac{L^+(\xi)+ L^-(\xi)}{2} & \mbox{ if Case 1(b) holds.}
	\end{cases}
\end{equation}

	\noindent

	\subsection{The trace is in the Hai\l{}asz-Sobolev space $M^{1,p}(\pom)$}
	\begin{remark}
		We keep $\lambda=3$ and the conventions of Remark \ref{r:lambda=3} in force.
	\end{remark}
	
	\begin{lemma}\label{l:trace-hajlasz-upper-grad}
		For $u \in \bV_{p, \lambda}$, let $f$ be defined as in \eqref{e:def-trace}. For each pair $\xi, \zeta$ satisfying the conclusions of Lemma \ref{l:bilateral-corona}, we have the estimate, 
		\begin{equation}\label{e:trace-sobolev}
			|f(\xi) - f(\zeta)| \lesssim |x-y| (\Nw(\grad u)(\xi) + \Nw (\grad u)(\zeta)). 
		\end{equation}
		The implicit constant is independent of $\xi, \zeta$.
	\end{lemma}
	
	\begin{remark}
		Lemma \ref{l:trace-hajlasz-upper-grad} (together with Lemma \ref{l:trace-in-Lp} below) corresponds to the first part of Theorem \ref{t:extension-in-text}(4).
	\end{remark}
	
	\begin{proof}
		Let $\ve>0$. Note that $\dT(\xi), \dT(\zeta) \neq \varnothing$. We consider various cases. 
		\begin{itemize}
			\item \textit{Both $\xi, \zeta$ satisfy Case 1(a)}. By letting $j$ sufficiently large, also depending on $\ve>0$, we find a pair of cube $Q,P \in \dD_{\sigma, j}$ so that $Q \in \dT(\xi)$, $P \in \dT(\zeta)$ and $|f(\xi) - \Phi(Q)| < \ve$ and $|f(\zeta) - \Phi(P)|$. We need to estimate $|\Phi(Q) - \Phi(S)|$. We have that $N_Q=N_P=1$. Hence $\Phi(Q)= u_{v_Q}$ and $\Phi(P)=u_{v_P}$. Moreover, $v_Q \in \cw(Q)$ and $v_P \in \cw(P)$, that is, $v_Q$ and $v_P$ are \textit{good} well connected components containing (half of) a good corkscrew ball coming from a compatible choice (recall the terminology in Remark \ref{r:compatible-choice}). So if $B^Q \in \dB(Q)$ and $B^P \in \dB(P)$, then $0.5B^Q \subset v_Q$ and $0.5B^P \subset v_P$. By Lemma \ref{l:compatibility}, we have that if $R$ is the minimal cube such that $x_{B^Q}, x_{B^P} \in 3B_R$, then there are $\theta$-carrot paths $\gamma(Q, R)$ and $\gamma(P, R)$ joining $x_{B^Q}$ to $x_{B^{3R}}$ and $x_{B^P}$ to $x_{B^{3R}}$, respectively, where $x_{B^{3R}}$ is the center of a good corkscrew ball of radius at least $3\theta \ell(R)$ contained in $3B_R \cap \om$. We then compute
			\begin{equation}\label{e:form10001}
				|u_{v_Q} - u_{v_P}| \leq |u_{v_Q} - u_{B^{3R}}| + |u_{B^{3R}}- u_{v_P}| = :A_1 + A_2. 
			\end{equation}
			We concentrate on $A_1$, as the estimate for $A_2$ follows the same path (no pun intended). Let $z \in 0.5B^Q$ be a point so that $u(z) <+\infty$. Then write
			\begin{equation*}
				A_1 \leq |u_{v_Q} - u(z)| + |u(z) + u_{B^{3R}}|. 
			\end{equation*}
			Note that since $z \in 0.5B^Q \subset v_Q$, we can connect any point $y \in v_Q$ to $z$ via a rectifiable path $\gamma(y,z) \subset v_Q$ of length $\lesssim_\theta |y-z|$. Hence
			\begin{align*}
				|u_{v_Q}- u(z)| &\leq \avint_{v_Q} |u(y) - u(z)| \, dy \leq \avint_{v_Q} \int_{\gamma(y, z)} \grad u(s) \, ds \, dy\\
				& \leq \avint_{v_Q} \int_{\gamma(y,z)} \sup_{s \in \gamma(y,z)} \grad u(s) \, ds\, dy \leq \ell(\gamma(y,z))\cdot \sup_{x \in v_Q} \grad u(x)\\
				& \lesssim \ell(Q) \Nw(\grad u)(\xi).
			\end{align*}
			Note that since $|\gamma(Q,R)| \lesssim \ell(R)$, then we can connect any point $y \in B^{3R}$ with a rectifiable curve $\gamma(y,z)$ with $|\gamma(y,z)|\lesssim \ell(R)$. Note also that $0.5B^{3R} \subset \Gamma_D^3(\xi) \cap \Gamma_D^3(\zeta)$. Thus
			\begin{align*}
				|u(z) - u_{B^{3R}}| & \leq \avint_{B^{3R}} |u(y) - u(z)| \,dy  \leq \avint_{B^{3R}} \int_{\gamma(y,z)} |\grad u|(s) \, ds\, dy\\
				& \lesssim  \avint_{0.5B^{3R}} \ell(R) \sup_{x \in \gamma(y,z)} \grad u(x) \, dy\\
				& \lesssim \ell(R) \Nw(\grad u)(\xi) \approx |\xi-\zeta| \Nw(\grad u)(\xi).
			\end{align*}
			Since the term $A_2$ in \eqref{e:form10001} can be estimated in a similar way, this gives \eqref{e:trace-sobolev} in the current case. 
			
			\item \textit{Both $\zeta$ and $\xi$ satify Case 1(b).} Recall that $\xi$ satisfy Case 1(b) if the limit $\lim_{Q \in \dT(\xi) \,: \, \ell(Q) \to 0} \Phi(Q)$ converges up to two sequences. In this case, we defined 
			\begin{equation*}
				f(\xi) = Tu(\xi) = \dfrac{L^+(\xi) + L^-(\xi)}{2},
			\end{equation*} 
			where $L^\pm(\xi)$ is the limit of the sequence $\{\Phi(Q)\}_{Q\in \dT^\pm(\xi)}$. For the $\ve>0$ as fixed at the beginning, pick $j^\pm(\xi)$ sufficiently small so that if $Q^\pm \in \dD_{\sigma, j^\pm(\xi)}$ with $Q^\pm \ni \xi$, then $|\Phi(Q^\pm) - L^\pm(\xi)|<\ve$. Choose $j^\pm(\zeta)$ in the analogue way. Then let $j\geq \max\{j^\pm(\xi), j^\pm(\zeta)\}$. From now on in this paragraph, we will only consider cubes with sidelength $\leq 2^{-j}$. Given $Q \in \dD_{\sigma, j}$ with $Q \ni \xi$, suppose without loss of generality that $v_{Q} \subset H^+(\xi)$. Denote $Q$ by $Q^+$. Let $P \in \dD_{\sigma, j}$ so that $P \ni \zeta$. Then either 
			\begin{align}
				& \tag{a} v_P \subset H^-(\zeta) \mbox{  or } \\
				& \tag{b} v_P \subset H^+(\zeta).
			\end{align} 
			Assume that the first instance (a) is the true one. Now we look at the pairs $(Q,P)$ with $\ell(Q)=\ell(P)$, $\xi \in Q$, $\zeta \in P$, and decreasing sidelength. Because both $\xi$ and $\zeta$ satisfy Case 1(b), it must happen that for some pair $(Q_1, P_1)$, we must have that
			\begin{align}
				& \tag{i} v_{Q_1} \subset H^-(\xi) \mbox{ or } \\
				& \tag{ii} v_{P_1} \subset H^-(\zeta).
			\end{align}  
			Assume that (i) (and recall of (a)) holds and that $v_{P_1} \subset H^+(\zeta)$. Then the conclusions of Lemma \ref{l:compatibility} hold for the two pairs of good corkscrews $(B^{Q^+}, B^{P})$ and  $(B^{Q_1}, B^{P_1})$. Since $B^Q \subset H^+(\xi)$, $B^P \subset H^-(\zeta)$, and $B^{Q_1} \subset H^-(\xi)$ and $B^{P_1} \subset H^-(\xi)$, we compute
			\begin{align}\label{e:form10002}
				& 	|L^+(\xi) + L^-(\xi) - L^+(\zeta) - L^-(\zeta)| \\
				& \leq |L^+(\xi) - \Phi(Q)| + |\Phi(Q) -\Phi(P)| + |\Phi(P) - L^+(\zeta)| \\
				& \quad + |L^-(\xi) - \Phi(Q_1)| + |\Phi(Q_1) - \Phi(P_1)| + |\Phi(P_1) - L^-(\zeta)|\\
				& < 4\ve +|\Phi(Q) -\Phi(P)| + |\Phi(Q_1) - \Phi(P_1)|. 
			\end{align}
			The last two term in the last display can be estimated as the term in \eqref{e:form10001}, to give
			\begin{equation}\label{e:form10003}
				|f(\xi)- f(\zeta)| \lesssim \ve + |\xi - \zeta|  \left( N_D^3(\xi) + N_D^3(\zeta), \right)
			\end{equation}
			which is the desired \eqref{e:trace-sobolev}. This, whenever (a), (i) and $v_{P_1} \subset H^-(\zeta)$ happens to holds. Now assume (a), (i) and $v_{P_1} \subset H^+(\zeta)$. We look at successive pairs of the same generation $(Q_2, P_2)$ so that $\xi \in Q_2 \subset Q_1$ and $\zeta \in P_2 \subset P_1$, and we look at the first one so that $v_{P_2}$ "switches" side, i.e. $v_{P_2} \subset H^+(\zeta)$. If $v_{Q_2} \subset H^-(\xi)$ and $v_{P_2} \subset H^+(\zeta)$, then we can carry out the computations \eqref{e:form10002} and \eqref{e:form10003} using $\Phi(Q)$, $\Phi(P)$, $\Phi(Q_2)$ and $\Phi(P_2)$. If, on the other hand, we have
			$v_{Q_2} \subset H^+(\xi)$, then we compute as in \eqref{e:form10002} and \eqref{e:form10003} with $\Phi(Q_1), \Phi(P_1)$ and $\Phi(Q_2)$ and $\Phi(P_2)$ instead. This let us conclude that the estimate \eqref{e:trace-sobolev} holds whenever we have (a), (i), and $v_{P_1} \subset H^+(\zeta)$. All the remaining cases can be dealt with  in a similar fashion. With the reader's help, we conclude that \eqref{e:trace-sobolev} holds whenever both $\xi$ and $\zeta$ satisfy Case 1(b).
			\item \textit{Both $\xi$ and $\zeta$ satisfy Case 2.} Recall that this is the case where there is a sequence of cubes $Q \in \dT(\xi)$ and $P\in \dT(\zeta)$ whose sidelength converges to $0$. Recall from \eqref{e:form20001} that 
			\begin{equation*}
				f(\xi) = \lim_{Q \in \dT(\xi) \,: \, \ell(Q) \to 0} \Phi(Q) = \frac{1}{2}\left( \lim_{v_Q \in \cw^+(\xi)} u_{v_Q} + \lim_{v_Q \in \cw^-(\xi)} u_{v_Q} \right),
			\end{equation*} 
			and similarly for $f(\zeta)$. 
			For the $\ve>0$ given above, choose $j$ so large so that for all cubes $Q \ni \xi$ and $P \ni \zeta$ with $\ell(Q) \leq 2^{-j}$, we have that 
			\begin{equation*}
				\left|\lim_{v_{Q'} \in \cw^+(\xi)} u_{v_{Q'}} - u_{v_Q^+}\right|< \ve \quad \mbox{ and }\quad  \left|\lim_{v_{Q'} \in \cw^-(\xi)} u_{v_{Q'}} - u_{v_Q^-}\right|< \ve,
			\end{equation*}
			where $v_{Q}^\pm$ is the good well connected component lying in $H^\pm(\xi)$. Choose $j$ so that the same is true for the limits in $\zeta$. Now, given a pair $(Q, P)$ with the properties above, let $B^Q_\pm \in \dB(Q)$ and $B^P_\pm \in \dB(P)$ (recall that $N_Q=N_P=2$). Then we see that $B^Q_+$ connects to $B^P_+$ (in the sense of Lemma \ref{l:compatibility}), or to $B^P_-$. In the first case, we must have that either $B^Q_-$ connects to $B^P_-$, or that $B^Q_+$ also connects to $B^P_-$ (otherwise $B^P_-$ wouldn't belong to $\dB(P)$, as a compatible choice, in the sense of Lemma \ref{r:compatible-choice}). We then may argue as in the computations \eqref{e:form10002}, and obtain \eqref{e:form10003}. The other case can be dealt with in the same way, and, again, we leave the details to the reader. 
			
			\item \textit{$\xi$ satisfies Case 1(a) and $\zeta$ satisfies Case 1(b)}. Let us give just a brief sketch: pick $j$ large enough, also depending on $\ve$, in particular so that, if $\ell(Q)\leq 2^{-j}$, then $|f(\xi) - u_{v_{Q}}|< \ve$; if $P^\pm \in \dT^\pm(\zeta)$ is so that $\ell(P^\pm)\leq 2^{-j}$ then $|L^\pm(\xi) - u_{v_{P^+}}|<\ve$. For one such triple $(Q,P^+,P^-)$, we may assume without loss of generality that $\ell(Q)=\ell(P^+)$ and that $\ell(P^-)< \ell(P^+)$. It suffices to estimate the difference
			\begin{equation*}
				u_{v_Q} - \frac{1}{2}u_{v_{P^+}} - \frac{1}{2} u_{v_{P^-}} = \frac{1}{2} \left(u_{v_Q} - u_{v_{P^+}} \right) + \frac{1}{2} \left(u_{v_Q} - u_{v_{P^-}} \right).
			\end{equation*}
			The first term on the right hand side of the latest display may be estimated as in \eqref{e:form10001} and the three displays below that. As for the second one, let $Q' \subset Q$ be a descendant of $Q$ containing $\xi$ and such that $\ell(Q')= \ell(P^-)$. Note that since $\xi$ satisfies case 1(a), and that $j$ was chosen sufficiently large, we have that $0.5 B^{Q'} \subset v_{Q'} \subset H^+(\xi)$, where $B^{Q'} \in \dB(Q')$, which is a corkscrew ball belonging to a compatible choice (see Remark \ref{r:compatible-choice}). Thus we can join $x_{B^Q}$ with $x_{B^{Q'}}$ with a $\theta$-carrot path in $\om$ with length $\lesssim \ell(Q)$. We then write
			\begin{equation*}
				(u_{v_Q} - u_{v_{P^-}}) = (u_{v_Q} - u_{v_{Q'}}) + (u_{v_{Q'}} - u_{v_{P^-}}).
			\end{equation*}
			Once again, the two terms on the right hand side are easily estimated as in \eqref{e:form10001} and the three display below it.
			\item \textit{$\xi$ satisfies Case 1(a) and $\zeta$ satisfies Case 2.}  This case is similar to the previous one, but in fact easier. We leave the details to the reader. 
			\item \textit{$\xi$ satisfies Case 1(b) and $\zeta$ satisfies Case 2.} Choose $j$ sufficiently large, depending also on $\ve>0$. Recall that, since $\xi$ satisfies Case 1(b), 
			\begin{equation*}
				f(\xi) = \frac{L^+(\xi) + L^-(\xi)}{2},
			\end{equation*}
			and that, since $\zeta$ satisfies Case 2, 
			\begin{equation*}
				f(\zeta)= \frac{1}{2} \lim_{v_P \in \cw^+(\xi) \atop \ell(P) \to 0} u_{v_{P}} + \frac{1}{2}\lim_{v_P \in \cw^-(\xi) \atop \ell(P) \to 0} u_{v_{P}}
			\end{equation*}
			Then we pick $j$ so large so that we can find a triple $(Q^+, Q^-, P)$ of cubes with sidelength $\leq 2^{-j}$, $\ell(Q^+) = \ell(P)$, and such that 
			\begin{equation*}
				|L^+(\xi) - \Phi(Q^+)|< \ve, \quad |L^-(\xi) - \Phi(Q^-)| < \ve
			\end{equation*}
			and
			\begin{equation*}
				\left|\lim_{v_{P'} \in \cw^+(\zeta) \atop \ell(P') \to 0} u_{v_{P'}} - u_{v_P^+}\right|< \ve, \quad 	\left|\lim_{v_{P'} \in \cw^-(\zeta) \atop \ell(P') \to 0} u_{v_{P'}} - u_{v_P^-}\right|,
			\end{equation*}
			where $v^\pm_P \subset \cw^\pm(\zeta)$. It then suffices to estimate
			\begin{equation}\label{e:form10006}
				\left|\frac{1}{2}u_{v_{Q^+}} + \frac{1}{2}u_{v_{Q^-}} - \frac{1}{2}u_{v_P^+} - \frac{1}{2}u_{v_P^-}\right|.
			\end{equation}
			Note that, by Lemma \ref{l:compatibility}, $B^{Q^+} \in\dB(Q^+)$ is connected with a $\theta$-carrot path to both $B^{P}_\pm$. On the other hand, let $P' \subset P$ be a descendant of $P$ containing $\zeta$ and such that $\ell(P') = \ell(Q^-)$. Then there is at least one good corkscrew ball $B^{P'} \in \dB(P')$. By virtue of being a good corkscrew ball belonging to a compatible choice, we see that $B^{P'}$ is connected by a $\theta$-carrot path to $B^{Q^-}$ (recalling also that $N_{Q^-}=1$). Without loss of generality, suppose that $B^{P'} \subset H^+$ (the argument is symmetric). Then $v_{P}^+$ and $B^{P'}$ are connected in the usual sense. We then split \eqref{e:form10006} as 
			\begin{equation*}
				\frac{1}{2}\left|u_{v_{Q^+}}- u_{v_P^-}\right| + \frac{1}{2}\left|u_{v_{Q^-}} - u_{B^{P'}}\right| + \frac{1}{2}\left|u_{B^{P'}} - u_{v_{P}^+}\right|.
			\end{equation*}
			As usual, these terms can be estimated as in \eqref{e:form10001} and the subsequent three displays. 
		\end{itemize}
		Since we considered all the possible cases (note that we can always switch the roles of $\xi$ and $\zeta$), we conclude the proof of the lemma.
	\end{proof}

	\begin{remark}
		Lemma \ref{l:trace-hajlasz-upper-grad} says that the limit function $f$ we obtained as trace of $u$ has an Hajlasz upper gradient. To conclude the proof of our trace theorem, we need to check that $f \in L^p(\pom)$. 
	\end{remark}
	
	\begin{lemma}\label{l:trace-in-Lp}
		$f \in L^p(\pom)$. 
	\end{lemma}
	\begin{proof}
		Pick $\xi_0 \in \pom$ such that $N(\grad u)(\xi_0) <+\infty$, and such that $|f(\xi_0)|<\infty$ is well defined as a limit. Then, by Lemma \ref{l:trace-hajlasz-upper-grad}, we have that
		\begin{equation*}
			|f(\xi)-f(\xi)| \lesssim |\xi-\xi_0| (\Nw(\grad u)(\xi) + \Nw(\grad u)(\xi_0)).
		\end{equation*}
		Thus 
		\begin{align*}
			\|f-f(\xi_0)\|_p^p \lesssim \diam(\om)^p\left( \int_\pom |N(\grad u)(\xi)|^p \, d\sigma(\xi) + N(\grad u)(\xi_0)^p \sigma(\pom)\right) < + \infty. 
		\end{align*}
	\end{proof}
	\begin{remark}
		Since $f \in L^p(\pom)$ and $\Nw(\grad u) \in L^p(\pom)$ is an Hajlasz uppser gradient of $f$, we conclude that $f \in M^{1,p}(\pom)$. Equation \eqref{e:trace-bound} follows now immediately:
		\begin{equation*}
			\|\dT u\|_{\dot M^{1,p}(\pom)} \leq \|N( \grad u)\|_{L^p(\pom)}.
		\end{equation*}
		This concludes the proof of Theorem \ref{t:extension-in-text}(4).
	\end{remark}
	
	\section{Construction of the extension}
	In this section, we gives proofs for Theorem \ref{t:extension-in-text}(1,3,5).
	\subsection{Definition of $u$}
	We start off by constructing the extension $u$.
	Suppose that $\pom$ is Ahlfors $d$-regular and consider the dyadic lattice $\DD_\sigma$ defined in Section \ref{s:cubes}.
	Then, for each Whitney cube $P\in \WW(\Omega)$ (as in Section \ref{ss:whitney}) there is some cube $Q\in\DD_\sigma$ such that $\ell(Q)\approx\ell(P)$
	and $\dist(P,Q)\approx \ell(Q)$, with the implicit constants depending on the parameters of $\DD_\sigma$ and on the Whitney decomposition. 
	\begin{align}\label{e-??}
		\mbox{Given } \, P \in \WW(\om), \, & \mbox{ we denote by  }\, Q_P \mbox{ the cube } Q \in \dD_\sigma \notag \\
		& \mbox{ s.t. } \ell(Q) \approx \ell(P) \mbox{ and } \dist(Q,P) \approx \ell(P).
	\end{align}
	Conversely,
	given $Q\in\DD_\sigma$, we let 
	\begin{equation}\label{eqwq00-b}
		w(Q) = \bigcup_{P\in\WW(\Omega) \atop Q_P=Q} P.
	\end{equation}
	It is immediate to check that $w(Q)$ is made up at most of a uniformly bounded number of cubes $P$, but it may happen
	that $w(Q)=\varnothing$.\\
	
	\noindent
	Next, for each Whitney cube $P\in \WW(\Omega)$ we consider 
	a $C^\infty$ bump function $\varphi_P$ supported on $1.1P$ such that the functions $\{\varphi_P\}_{P\in\WW(\Omega)}$,
	form a partition of unity of $\chi_\Omega$. That is,
	\begin{equation*}
		\sum_{P\in\WW(\Omega)}\varphi_P = \chi_\Omega.
	\end{equation*}
	Also, if we set
	\begin{equation}\label{e:def-A_P}
		A_P := A_{M\cdot B_{Q_P}},
	\end{equation}
	where $A_{M\cdot B_{Q_P}}$ is the affine map as in Lemma \ref{l:A_B} - roughly speaking the affine map that minimises $\gamma_f^1(M\cdot B_{b(P)})$, and $M$ is a sufficiently large constant, as in Proposition 9.9 of \cite{AMV1}. Finally,
	we define the extension $u:\overline \Omega\to\R$ of $f$ as follows:
	\begin{align}\label{eq:extension}
		& u|_{\pom} = f \nonumber \\
		& u|_\Omega =  \sum_{P\in\WW(\Omega)} \varphi_P\,A_{P}.
	\end{align}
	We note immediately that $u$ is smooth in $\Omega$.
	
	Below we will need the following fact, which we recall from Lemma 5.18 of \cite{AMV1}.
	\begin{itemize}
		\item If $B$ is a ball centered on $\pom$, then 
		\begin{equation} \label{ext-form1}
			|\grad A_{B'}| \lesssim \gamma^1_f(B) + \left( \avint_{5\Lambda B} g^s \, d \sigma \right)^{\frac{1}{s}},
		\end{equation}
		where $g \in \Grad_p(f)$ and $s\leq p$.
		\item If $B \subset B'$ are two balls centered on $\pom$ so that $r_B \approx_c r_{B'}$, then 
		\begin{equation}\label{ext:form2}
			|A_B(x) - A_{B'}(x)| \lesssim_c \gamma_f^1(B') \left( \dist(x, B') + r_{B'}\right), 
		\end{equation}
		and also
		\begin{equation}\label{ext:form2b}
			|\grad A_B - \grad A_{B'}| \lesssim \gamma_f^1(B'). 
		\end{equation}
	\end{itemize}
	Furthermore, from Lemma 5.17 of \cite{AMV1} 
	 we have that, if again $B$ is centered on $\pom$, then
	\begin{equation}\label{ext:form3}
		\gamma_f^1(B) \lesssim \left( \avint_{5 \Lambda B} g^s\, d \sigma\right)^{\frac{1}{s}},
	\end{equation}
	whenever $g \in \Grad_p(f)$ and $1\leq s \leq p$. \\

	\subsection{A technical lemma}
	In this subsection we prove the following lemma.
	\begin{lemma}\label{lem46}
		There exists a constant $C\geq 1$, depending\footnote{Recall that $\Lambda$ is the constant from Theorem \ref{t:Hajlasz-poincare}.} on $M$, $\Lambda$, and the parameters of the Whitney and boundary lattices, so that, for each $P_0 \in \dW(\Omega)$, if $x \in P_0$, then
		\begin{align}
			&  |\nabla u(x)| \lesssim \left( \avint_{5\Lambda C B_{Q_{P_0}}} g^s\, d \sigma \right)^{\frac{1}{s}}, \,\, ;1 \leq s \leq p; \label{e:ext-lem1}\\
			& |\nabla^2 u(x)|\lesssim \gamma_{f}(C B_{b(P_0)})\,\ell(P_0)^{-1}.\label{e:ext-lem2}
		\end{align}
	\end{lemma}
	
	\begin{proof}
		Now let $x\in P_0$. Note that
		\begin{align*}
			\grad u(x) & = \grad \left( u(x) - A_{P_0}(x)) + \grad A_{P_0}(x) \right)\\ 
			& = \sum_{P \in \dW(\Omega)} \grad \varphi_P(x) \left(A_P(x)-A_{P_0}(x)\right) + \sum_{P \in \dW(\om)} \varphi_P(x) \grad A_P (x).
		\end{align*}
		Then, 
		\begin{align*}
			|\grad u (x) | \leq \sum_{P \in \dW(\om)} |\grad \varphi_P(x)| |A_P(x) - A_{P_0}(x)| + \sum_{ P \in \dW(\om)} |\varphi_P(x)| |\grad A_P(x)| = S_1 + S_2. 
		\end{align*}
		We estimate the first sum $S_1$. Since $x \in P_0$, then $\grad \varphi_P(x) \neq 0$ only whenever $1.1P \cap P_0 \neq \emptyset$. Thus, for these Whitney cubes $P$, $\ell(P) \approx \ell(P_0)$, where the implicit constant depend (only) on the construction of the Whitney cubes. In particular, this implies that $\ell(Q_P) \approx \ell(Q_{P_0})$. Since $\dist(Q_P, P) \approx \ell(P)$, it is also easy to see that $\dist(Q_P, Q_{P_0}) \lesssim \ell(Q_{P_0}) \approx \ell(P_0)$. We conclude that there exists a constant $C \geq 1$ so that 
		\begin{equation*}
			CB_{Q_{P_0}} \supset 2 B_{Q_P} \, \mbox{ for each }\, P \in \dW(\om)  \, \mbox{ such that } \, 1.1P \cap P_0 \neq \emptyset.
		\end{equation*}
		Moreover, for obvious reasons, $\#\{P \in \dW(\om) \, |\, 1.1P \cap P_0 \neq \emptyset \} \lesssim_d 1$. Using \eqref{ext:form2} and recalling that $\grad \varphi_P \lesssim \ell(P)^{-1}$, we compute
		\begin{align*}
			S_1 & \lesssim_{d,C} \frac{\gamma_f^1(C B_{Q_{P_0}}) \left( \dist(x, C B_{Q_{P_0}}) + C' \ell(P_0)\right) }{\ell(P_0)}\\
			& \lesssim \gamma_f^1(C B_{Q_{P_0}})\\
			& \stackrel{\eqref{ext:form3}}{\lesssim} \left( \avint_{5 \Lambda C B_{Q_{P_0}}} g^s\, d \sigma \right)^{\frac{1}{s}},
		\end{align*}
		for $1 \leq s \leq p$. 
		This takes care of $S_1$. The bound
		\begin{equation*}
			S_2 \lesssim \left( \avint_{5 \Lambda C B_{Q_{P_0}}} g^s\, d \sigma \right)^{\frac{1}{s}}
		\end{equation*}
		follows via a similar reasoning and using \eqref{ext-form1} and \eqref{ext:form3}. This and the preceeding estimate prove \eqref{e:ext-lem1}.\\
		
		\noindent
		We now turn our attention to \eqref{e:ext-lem2}. Set $\d_i=\d_{x_i}$, $1 \leq i \leq d+1$. Recalling that $\sum_P \varphi(x) = 1$, we compute 
		\begin{align*}
			\sum_P \varphi_P(x) \d_jA_P(x) = \sum_P \varphi_P(x) \d_j(A_P-A_{P_0})(x) - \d_jA_{P_0}(x).
		\end{align*}
		Since $A_{P_0}$ and $A_P$ are affine, we get
		\begin{align*}
			\d_i\left(\sum_P \varphi_P(x) \d_jA_P(x) \right) = \sum_{P} \d_i \varphi_P(x) \, \d_j (A_P - A_{P_0})(x),
		\end{align*}
		and hence
		\begin{align*}
			\d_i \d_j u(x) & = \d_i \left[ \sum_{P} \d_j \varphi_P(x) \left(A_P(x) - A_{P_0}(x)\right) + \sum_P \varphi_P(x) \d_j A_P(x) \right] \\
			& = \sum_{P}\d_i \d_j \varphi_P(x) \left(A_P(x) - A_{P_0}(x)\right) + \sum_P \d_j \varphi_P(x) \, \d_i \left( A_P - A_{P_0} \right)(x) \\
			& + \sum_{P} \d_i \varphi_P(x) \, \d_j (A_P - A_{P_0})(x)\\
			& =: S_1 + S_2 + S_3.
		\end{align*}
		Using that $|\grad^2 \varphi_P| \lesssim \ell(P)^{-2}$ and \eqref{ext:form2}, and reasoning as above, we obtain that $|S_1| \lesssim \gamma_f^1(C B_{Q_{P_0}}) \ell(P_0)^{-1}$. Similarly, using \eqref{ext:form2b}, we have $|S_2| \lesssim \gamma_f^1(CB_{Q_{P_0}}) \ell(P_0)^{-1}$ and also $|S_3| \lesssim \gamma_f^1(CB_{Q_{P_0}}) \ell(P_0)^{-1}$.
	\end{proof}

	\subsection{Proof of Theorem \ref{t:extension-in-text}(1): the $L^p$ estimates \eqref{e:extension-Lp-est}}
	To ease some computations, we introduce the following dyadic versions of $\dN$ and $\dS$. For a point $x \in \pom$, we set
	\begin{equation}\label{e:dyadic-cone}
		\Gamma_D(x) = \bigcup_{Q \in \dD_\sigma \atop Q \ni x} w(Q) = \bigcup_{Q \in \dD_\sigma \atop Q \ni x} \bigcup_{P \in \dW(\om) \atop Q = b(P)} P.
	\end{equation}
	This is a `dyadic' version of the cone $\Gamma(x)$. Then, for possibly vector valued function $F$ on $\Omega$, we put
	\begin{equation*}
		\dN_D(F)(x) := \sup_{y \in \Gamma_D(x)} |F(y)|,
	\end{equation*}
	and 
	\begin{equation*}
		\dS_D(F)(x) := \left( \int_{\Gamma_D(x)} \frac{|F(y)|^2}{d_\Omega(y)^{d+1}} \, dy \right)^{\frac{1}{2}}. 
	\end{equation*}
	\begin{lemma}
		Assumptions and notation as in Theorem \ref{t:extension-in-text}. It suffices to show the estimate \eqref{e:extension-Lp-est} with $\dN_D$ and $\dS_D$.
	\end{lemma}
	\begin{proof}
		It is not difficult to see that $\|\dN(\grad u)\|_{L^p(\pom)} \approx \|\dN_D (\grad u)\|_{L^p(\pom)}$, and similarly for $\dS_D$ and $\dS$. See \cite{hofmann2020uniform}, Lemma 1.27 and its proof. 
	\end{proof}
	\begin{lemma}
		With assumptions and notation of Theorem \ref{t:extension-in-text}, for any $g \in \Grad_p(f)$,
		\begin{equation*}
			\|\dN(\grad u) \|_{L^p(\pom)} \lesssim \|g\|_{L^p(\pom)}.
		\end{equation*}
	\end{lemma}
	\begin{proof}
		Let $x \in \pom$ and $y \in \Gamma_D(x)$. Then there is a $P=P(y) \in \dW(\om)$ such that $Q_P \in \dD_\sigma$ has center $x_{Q_P}$ satisfying $|x-x_{Q_P}| \lesssim_{C'} \ell(Q_P)\approx \ell(P)$, where the implicit constant $C'$ depends only on the Whitney decomposition and on the parameters of $\dD_\sigma$. Hence, 
		\begin{equation*}
			|\grad u(y)| \stackrel{\eqref{e:ext-lem1}}{\lesssim}\left(\avint_{5 \Lambda CB_{Q_{P(y)}}} g^s \, d \sigma \right)^{\frac{1}{s}} \lesssim \left( \avint_{C'' B(x,\ell(P(y)))} g^s \, d \sigma \right)^{\frac{1}{s}},
		\end{equation*}
		where $C''$ depends on $C'$, $C$ and $\Lambda$. Clearly then, for any $y \in \Gamma_D(x)$,
		\begin{equation}
			|\grad u(y)| \lesssim \dM^s f (x).
		\end{equation}
		Thus we conclude that
		\begin{align*}
			\int_{\pom} |\dN (\grad u)(x) |^p \, d\sigma(x) & \leq C \int_{\pom} |\dM^s f (x)|^p \, d \sigma(x)\\
			& = C \int_{\pom} (M_{C \diam(\om)} (g^s) )^{\frac{p}{s}} \, d \sigma \\
			& \lesssim_{p/s} \int_{\pom} g^p \, d \sigma. 
		\end{align*}
	\end{proof}
	
	\begin{lemma}
		With the assumptions and notation of Theorem \ref{t:extension-in-text}, we have 
		\begin{equation*}
			\|\dS(d_\Omega(\cdot) \, \grad^2 u)\|_{L^p(\pom)} \lesssim \|g\|_{L^{p}(\pom)}.
		\end{equation*}
	\end{lemma}
	\begin{proof}
		We write 
		\begin{align*}
			\|\dS(d_\Omega(\cdot) \grad^2u)\|^p_{L^p(\pom)}&  = \int_{\pom} \left| \int_{\Gamma_D(x)} \frac{|\grad^2 u(y)|^2 d_\Omega(y)^2}{d_\Omega(y)^{d+1}} \, dy \right|^{\frac{p}{2}} \, d\sigma(x)\\
			& \leq \int_{\pom} \left| \sum_{Q \ni x} \sum_{P \in w(Q)} \int_{P} \frac{|\grad^2 u(y)|^2 d_\Omega(y)^2}{d_\Omega(y)^{d+1}} \, dy \right|^{\frac{p}{2}} \, d \sigma(x)\\
			& \stackrel{\eqref{e:ext-lem2}}{\lesssim} \int_{\pom} \left| \sum_{Q \ni x} \sum_{P \in w(Q)} \int_P  \frac{\gamma_f^1(CB_{Q_P})^2}{\ell(Q_P)^{d+1}} \, dy \right|^{\frac{p}{2}} \, d \sigma(x).
		\end{align*}
		It is immediate from the definitions that if $Q \in \dD_\sigma$ and $P \in w(Q)$, then $\ell(Q_P) \approx \ell(Q)$ and $\dist(Q_P, Q)\lesssim \ell(Q)$. Hence there is a constant $C' \geq 1$, depending on the constant $C$ in the last display, and possibly the parameters of the Whitney cubes, so that $\gamma_f^1(CB_{Q_P}) \lesssim \gamma_f^1(C' B_Q)$. We can then conclude that
		\begin{align*}
			\|\dS(d_{\om}(\cdot) \grad^2 u) \|_{L^p(\pom)}^p & \lesssim \int_{\pom} \left| \sum_{Q \ni x} \gamma^1_f(C'B_Q)^2 \right|^{\frac{p}{2}} \, d \dH^d(x) \\
			& \approx \| \dG_f^1 \|_{L^{p}(\pom)}^p \lesssim \|g\|_{L^p(\pom)}^p,
		\end{align*}
		where the last inequality is Theorem A of \cite{AMV1}.
	\end{proof}

	\subsection{Proof of Theorem \ref{t:extension-in-text}(3): convergence in the Lipschitz case}
	We now prove Theorem \ref{t:extension-in-text}(3) and in the next subsection we prove Theorem \ref{t:extension-in-text}(5), that is, the non-tangential convergence almost everywhere of the extension $u$ to $f$.\\

	\noindent
	Hypotheses as in Theorem \ref{t:extension-in-text}(3).
	\begin{remark}\label{r:u-in-lip-bar-om}
		The fact that $u \in \Lip(\overline{\Omega})$, equation \eqref{e:lip-ext} is proven in \cite[Lemma 4.2]{mourgoglou2021regularity}. The extension is denoted there by $\widetilde{f}$, but the definition is precisely the one used here. 
	\end{remark}
	
	\noindent
	It remains to show the non-tangential convergence of $\grad u$, i.e. \eqref{ext-limit}. To this end, we briefly recall Proposition 9.9 of \cite{AMV1}. 
	 If $\partial \Omega$ is uniformly $d$-rectifiable, $M>1$ sufficiently large and $f : \pom \to \R$ is $L$-Lipschitz, the, for $\sigma$-almost all $x \in \pom$, the following holds. If we take a sequence of cubes $\{Q_j\}_{j=1}^\infty$ such that $\ell(Q_j)\to 0$ and $x \in Q_j$ for each $j$, then
	\begin{equation}\label{ext:form5}
		\lim_{j \to \infty} \frac{\sup_{y \in MB_{Q_j}} \left[ A_{Q_j}(y) - f(x) - \langle \grad_t f(x), y- x \rangle \right] }{\ell(Q_j)} = 0.
	\end{equation}
	
	\noindent
	Now let $x \in \partial \Omega$ be so that \eqref{ext:form5} holds and let $y \in \Gamma_D(x)$. Since $y\to x$ in $\Gamma_D(x)$, then there is a sequence $\{P_j\}$ of Whitney cubes so that $P_j \subset \Gamma_D(x)$, $\ell(P_j) \to 0$ and $\dist(x, P_j) \to 0$. The corresponding boundary cubes $\{Q_j\} \subset \dD_\sigma$ satisfy the hypotheses of Proposition 9.9 of \cite{AMV1}. 
	For $y \in \Gamma_D(x)$, call $P(y)$ the Whitney cube containing it. We then compute
	\begin{align*}
		|\grad u(y)- \grad_t f(x)| & \leq \left| \sum_P \grad \varphi_P(y) \left( A_P(y) - A_{P(y)}(y) \right) \right| + \left| \sum_{P} \varphi_P(y) \left(\grad A_{P} - \grad_t f(x) \right) \right|\\
		& := I_1 + I_2. 
	\end{align*}
	Using \eqref{ext:form2b}, the fact that $\varphi_P(y) \neq 0$ only whenever $1.1P \cap P(y) \neq \emptyset$, and $\grad \varphi_P(y) \lesssim \ell(P)^{-1} \approx \ell(P(y))^{-1}$, we immediately obtian that $I_1 \lesssim \gamma_f^1(CB_{Q_{P(y)}})$. On the other hand, 
	\begin{align*}
		I_2 & \lesssim \left| \sum_{P} \varphi_P(y)(\grad A_P - \grad A_{P(y)}) \right| + \left| \sum_P \varphi_P(y) ( \grad A_{P(y)} - \grad_t f(x)) \right| \\
		& := I_{2,1} + I_{2,2}.
	\end{align*}
	The term $I_{2,1}$ is dealt with just like $I_1$,  using \eqref{ext:form2} instead of \eqref{ext:form2b}. As for the other one, using Lemma 9.11 of \cite{AMV1},
	\begin{equation*}
		I_{2,2} \lesssim |\grad A_{P(y)} - \grad_t f(x)| \to 0 
	\end{equation*}
	as $y \to x$ in $\Gamma_D(x)$. This concludes the proof of the non-tangential convergence of $\grad u$ and thus that of Theorem \ref{t:extension-in-text}(3).

	\subsection{Non-tangential convergence of $u$ to $f$ for general $f$}
	What is left is the proof Theorem \ref{t:extension-in-text}(5). The next lemma takes care of this. 
	\begin{lemma}\label{l:lemma-convergence}
		Hypotheses as in Theorem \ref{t:extension-in-text}(5). Then 
		\begin{equation}
			u \to f(\xi) \, \, \mbox{non-tangentially a.e.} \,\, \xi \in \pom.
		\end{equation}	
	\end{lemma}
	
	\subsubsection{Assume first that $f \in \Lip(\pom) \cap M^{1,p}(\pom)$}\
	\noindent
	The statement in general will follow by density of $\Lip(\pom) \cap M^{1,p}(\pom)$ in $M^{1,p}(\pom)$ (see \cite[Lemma 10.2.7]{heinonen2015sobolev}). 
	But if $f \in \Lip(\pom)$, then $u \in \Lip(\overline{\om})$ by Remark \ref{r:u-in-lip-bar-om}. Thus non-tangential convergence holds by continuity.

	\noindent
	\textit{General case: let $f \in \M^{1,p}(\pom)$}. 
	For a sequence $\lambda_k \to \infty$ as $k \to \infty$, set
	\begin{equation}
		E_k:= \{ x \in E \, |\, g(x) > \lambda_k\}. 
	\end{equation}
	Note that $f|_{E\setminus E_k}$ is $C\lambda_k$-Lipschitz. We denote by $f_k$ the standard $C\lambda_k$-Lipschitz extension to the whole of $E$ of $f|_{E \setminus E_k}$. It is shown in the proof of \cite[Lemma 10.2.7]{heinonen2015sobolev} (but see also the original paper \cite[Theorem 5]{hajlasz1996sobolev}) that 
	\begin{equation}
		g_k = (g- C\lambda_k)\mathds{1}_{E_k} \in \Grad_p(f_k).
	\end{equation}
	We extend $f_k$ to $\om$ as in \eqref{eq:extension}, and call the extension $u_k$.\\

	\noindent
	Let $ x_0 \in G$, where $G \subset E$ is a subset of full measure, i.e. $|E|=|G|$, which we will implicitly re-define several times below, so that all the statement that hold $\sigma$-a.e. in $E$, hold everywhere in $G$. To start with, we assume that $G$ is a full measure subset of $E \setminus \cap_{k \geq 1} E_k$, so that if $x_0 \in G$, then for some $k=k(x_0)$ sufficiently large, $\mathds{1}_{E_k}(x_0)=0$. Now let $\ve>0$. Since $\dM^s f(x) <+ \infty$ $\sigma$-a.e., we may assume that $\dM^s f(x_0) <+ \infty$, and hence we can find a $\delta=\delta(x_0)>0$ sufficiently small so that 
	\begin{equation}\label{ext-form602}
		\dM^s f(x_0) \cdot \delta < \ve. 
	\end{equation}
	Let $y \in \Gamma_D(x) \cap B(x,\delta)$. We denote by $A^k_P$ the best approxiating affine maps defining $u_k$ - they approximate $f_k$. Let also $P_y$ be the Whitney cube in $\dW(\om)$ containing $y$. We compute
	\begin{align*}
		|u_k(y) - u(y)| & = \left| \sum_P \vp(y)\left(A_{P}^k(y) - A_P(y) \right) \right| \\
		& \leq \left| \sum_P \vp(y) (A_P^k(y) - A^k_{P_y}(y)) \right| + \left| \sum_P \vp(y) (A^k_{P_y}(y) - A_{P_y}(y)) \right| \\
		& \quad \quad \quad + \left|\sum_P \vp(y) (A_{P_y}(y) - A_P(y)) \right|\\
		& := I_1+I_2 + I_3.
	\end{align*}
	Since $1.1P \cap P_y \neq \emptyset$, then $\ell(P) \approx P_y$, and therefore there is a $C \geq 1$, depending on $M$, so that  $M \cdot B_{Q_{P_y}}, M \cdot B_{Q_P} \subset B_{x_0,y}:= B(x_0, C \ell(P_y))$ - recall the definition of $A_P$ in \eqref{e:def-A_P}. Hence, using \eqref{ext:form2}, we have
	\begin{align}\label{ext-form505}
		I_1  \lesssim \sum_{P \in \WW(\om) \atop 1.1P \cap P \neq \varnothing} |A^k_P(y) - A^k_{P_y}(y)|
		\lesssim \gamma_{f_k}^1 (B_{x_0,y}) \ell(P(y)) \lesssim \avint_{5 \Lambda  B_{x_0,y}} g_k \, d \sigma \cdot \delta.
	\end{align}
	Proceeding analogously,
	\begin{equation}\label{ext-form601}
		I_3 \lesssim \gamma_f^1 (B_{x_0,y}) |x_0-y| \lesssim \dM^s f(x_0)\cdot  \delta < \ve. 
	\end{equation}
	Before going further, we want to bound \eqref{ext-form505}. Set $B=5 \Lambda B_{x_0, y}$, just for the present calculation. We see that 
	\begin{align*}
		\avint_B g_k \, d \sigma = \avint_B g_k - g \, d\sigma +  \avint g\, d\sigma. 
	\end{align*}
	Now, 
	\begin{align*}
		\avint_B g-g_k\, d\sigma & \lesssim \frac{1}{r(B)^d} \int_{B \cap E_k} g-g_k \, d \sigma + \frac{1}{r(B)^d} \int_{B \setminus E_k} g \, d \sigma \\
		& \lesssim \lambda_k \frac{|B \cap E_k|}{|B|} + \dM^s f (x_0).
	\end{align*}
	Thus
	\begin{equation}\label{ext-form600}
		I_1 \stackrel{\eqref{ext-form505}}{\lesssim} \delta \cdot \lambda_k \frac{|B \cap E_k|}{|B|} + \ve. 
	\end{equation}
	We are left with $I_2$. We compute
	\begin{align*}
		I_2 & \lesssim |A_{P_y}^k(y) - A_{P_y}(y)|\\
		& \lesssim |A_{P_y}^k(y) - A_{P_y}^k(x_0)| + |A_{P_y}^k(x_0) - A_{P_y}(x_0)| + |A_{P_y}(x_0) - A_{P_y}(y)| \\
		& =: I_{2,1} + I_{2,2} + I_{2,3}.
	\end{align*}
	Using again \eqref{ext:form2}, and reasoning as in the paragraph above \eqref{ext-form505}, we obtain that
	\begin{equation*}
		I_{2,1} \lesssim \delta \cdot  \avint_{5 \Lambda B_{x_0,y}} g_k \, d\sigma, 
	\end{equation*}
	and hence, as in \eqref{ext-form600},
	\begin{equation*}
		I_{2,1} \lesssim \delta \cdot \lambda_k \frac{|5 \Lambda B_{x_0,y} \cap E_k|}{|5 \Lambda B_{x_0, y}|} + \ve.
	\end{equation*}
	On the other hand, as for $I_3$ above, in \eqref{ext-form601}, we get
	\begin{equation*}
		I_{2,3} \lesssim \gamma_f^1(B_{x_0, y}) \cdot |x_0-y| \lesssim \dM^s f(x_0) \cdot \delta. 
	\end{equation*}
	To deal with $I_{2,2}$ we compute 
	\begin{align*}
		I_{2,2}& =|A_{P_y}^k(x_0) - A_{P_y}(x_0)|\\
		& \leq |A_{P_y}^k(x_0) - f_k(x_0)| + |f_k(x_0) - f(x_0)| + |f(x_0) - A_{P_y}(x_0)| \\
		& =: I_{2,2,1} + I_{2,2,2} + I_{2,2,3}.
	\end{align*}
	Assuming that we were working with a subsequence, as we may, we have that 
	\begin{equation*}
		I_{2,2,2} \to 0 \mbox{ as } k \to \infty.
	\end{equation*}
	Let us look at $I_{2,2,1}$. We have
	\begin{align*}
		I_{2,2,1}&  \leq\left| \avint_{B_{x_0, y}} A_{P_y}^k(x_0) - A_{P_y}^k(x) \, d \sigma(x) \right| \\
		& \quad \quad + \left| \avint_{B_{x_0, y}} A_{P_y}^k(x) - f_k(x) \, d \sigma(x) \right| + \left| \avint_{B_{x_0, y}} f_k(x) - f_k(x_0) \, d \sigma(x) \right|\\
		& =: I_{2,2,1}^1 + I_{2,2,1}^2 + I_{2,2,1}^3.
	\end{align*}
	Now, recalling that $r(B_{x_0,y}) \approx \ell(P_y) \approx |x_0-y| < C \delta$, 
	\begin{equation*}
		I_{2,2,1}^1 \lesssim \delta \cdot |\grad A_{P_y}^k|.
	\end{equation*}
	We use \eqref{ext-form1} to bound the right hand side of the display above, thus obtaining, 
	\begin{equation*}
		I_{2,2,1}^1 \lesssim \delta \cdot \avint_{5 \Lambda B_{x_0,y}} g_k \, d \sigma \lesssim \delta \cdot \lambda_k \frac{|5 \Lambda B_{x_0, y} \cap E_k|}{|5 \Lambda B_{x_0,y}|} + \ve,
	\end{equation*}
	where the last inequality follows as in \eqref{ext-form600}.
	On the other hand, 
	\begin{equation*}
		I_{2,2,1}^2 \lesssim \ell(P_y) \gamma_{f_k}^1(B_{x_0,y}) \lesssim \delta \cdot \lambda_k \frac{|5 \Lambda B_{x_0, y} \cap E_k|}{|5 \Lambda B_{x_0,y}|} + \ve,
	\end{equation*}
	where the last inequality is as in \eqref{ext-form505} and \eqref{ext-form600}. Finally, using the definition of Haj\l{}asz upper gradient, we find that 
	\begin{equation*}
		I_{2,2,1}^3 \lesssim \avint_{B_{x_0,y}} |x_0-x|(g(x)+g(x_0)) \, d \sigma(x) \lesssim \ell(P_y) \avint_{5 \Lambda B_{x_0,y}} g_k \, d \sigma + \ell(P_y) g_k(x_0).
	\end{equation*}
	If we assume $k$ sufficently large, then $g_k(x_0)=0$ (see the paragraph above \eqref{ext-form602}. The other term on the right hand side of the last display can be dealt with as in \eqref{ext-form505}-\eqref{ext-form600}. The term $I_{2,2,3}$ is dealt with in a similar manner.\\
	
	\noindent
	To conclude, for any $\ve>0$, for all $y \in \Gamma_{D}(x) \cap B(x_0, \delta)$, 
	\begin{equation*}
		|u_k(y) - u(y)| \lesssim \ve + \delta \cdot \lambda_k \frac{|B\cap E_k|}{|B|}.
	\end{equation*}
	It is show in the proof of \cite[Lemma 10.2.7]{heinonen2015sobolev} that 
	\begin{equation*}
		\lambda_k |E_k| \to 0 \mbox{ as } k \to \infty.
	\end{equation*}
	We can now conclude the proof of Lemma \ref{l:lemma-convergence}. Let $x_0 \in G E\setminus E_0$, where $E_0$ is a subset of $E$ with $|E_0|=0$, chosen so that all a.e. requirements in the computations above are true. For any $\ve>0$, we find a $\delta>0$ so that if $y \in \Gamma_D(x)\cap B(x_0, \delta)$, then 
	\begin{align*}
		|u(y)-f(x_0)| &\leq |u(y)-u_k(y)| + |u_k(y)- f_k(x_0)| + |f_k(x_0)-f(x_0)| \\
		& < 2 \ve + \delta \cdot \lambda_k \frac{|B\cap E_k|}{|B|} + |f_k(x_0)-f(x_0)|.
	\end{align*}
	Letting $k\to \infty$, we are left with 
	\begin{equation}
		|u(y)-f(x_0)|  < 2\ve, 
	\end{equation}
	and this ends the proof of Lemma \ref{l:lemma-convergence}, and thus of Theorem \ref{t:extension-in-text}(5).

	\newpage

	\bibliography{bibliography}
	\bibliographystyle{halpha-abbrv}

\end{document}